\documentclass[10pt,reqno]{amsart}
\usepackage{amssymb}
\usepackage{amsmath}
\usepackage{amsfonts}
\usepackage{hyperref}
\usepackage[title]{appendix}
\usepackage[all,poly]{xy}
\usepackage{enumerate}
\usepackage[square,sort,comma,numbers]{natbib}
\usepackage{extpfeil}
\usepackage{color}
\usepackage{lipsum}
\usepackage{tabularx}
\usepackage{bm}
\usepackage{enumitem}
\usepackage{CJKutf8}
\allowdisplaybreaks

\begin{document}
\begin{CJK}{UTF8}{gbsn}
    \theoremstyle{plain}
    \newtheorem{thm}{Theorem}[section]
    \newtheorem{theorem}[thm]{Theorem}
    \newtheorem{lemma}[thm]{Lemma}
    \newtheorem{corollary}[thm]{Corollary}
    \newtheorem{corollary*}[thm]{Corollary*}
    \newtheorem{proposition}[thm]{Proposition}
    \newtheorem{proposition*}[thm]{Proposition*}
    \newtheorem{conjecture}[thm]{Conjecture}
    %%%%%%%%%%%%%%%%%%%% Text roman %%%%%%%%%%%%%%%%%%%%%%%%%%%%%
    \theoremstyle{definition}
    \newtheorem{construction}[thm]{Construction}
    \newtheorem{notations}[thm]{Notations}
    \newtheorem{question}[thm]{Question}
    \newtheorem{problem}[thm]{Problem}
    \newtheorem{remark}[thm]{Remark}
    \newtheorem{remarks}[thm]{Remarks}
    \newtheorem{definition}[thm]{Definition}
    \newtheorem{claim}[thm]{Claim}
    \newtheorem{assumption}[thm]{Assumption}
    \newtheorem{assumptions}[thm]{Assumptions}
    \newtheorem{properties}[thm]{Properties}
    \newtheorem{example}[thm]{Example}
    \newtheorem{comments}[thm]{Comments}
    \newtheorem{blank}[thm]{}
    \newtheorem{observation}[thm]{Observation}
    \newtheorem{defn-thm}[thm]{Definition-Theorem}
		\newcommand{\Rmnum}[1]{\uppercase\expandafter{\romannumeral #1}}  
		\newcommand{\rmnum}[1]{\romannumeral #1}

\def\vol{\operatorname{vol}}

    %%%%%%%%%%%%%%%%%%%%%%%%%%%%%%%%%%%%%%%%%%%%%%%%%%%%%%%%%%%%%%

    \title{Large genus asymptotics of super Weil-Petersson volumes}
    \author[X. Huang]{Xuanyu Huang}
    \address{Center of Mathematical Sciences, Zhejiang University, Hangzhou, Zhejiang 310027, China}
    \email{Hxuanyu98@gmail.com}
    
    \maketitle
\begin{abstract}
In this paper, we obtain the asymptotic expansions of super intersection numbers and prove that the associated coefficients are polynomials. Moreover, we give an algorithm which can explicitly compute these coefficients. As an application, we prove the existence of a complete asymptotic expansion of super Weil-Petersson volumes in the large genus. This generalizes the celebrated work of Mirzakhani-Zograf. We also confirm two conjectural formulae proposed by Griguolo-Papalini-Russo-Seminara.
\end{abstract}

\section{Introduction}
\subsection{Super Weil-Petersson volumes} JT supergravity is a generalization of JT gravity with supersymmetry in the boundary. Their partition functions can be expressed in terms of Weil-Petersson volumes of moduli space of super Riemann surface and moduli space of curves, respectively \cite{kimura2020jt,stanford2020jt}. Denote $\hat{V}_{g,n}(L_1,...,L_n)$ as the super Weil-Petersson volume with boundary components of lengths $L_1,...,L_n$. Stanford-Witten \cite{stanford2020jt} proved that such volume is the integral of the Euler form of a holomorphic vector bundle $E_{g,n}^{\vee}$ combined with the Weil-Petersson symplectic form over the moduli space of spin hyperbolic surfaces which is denoted by $M_{g,n,\vec{o}}^{spin}(L_1,...,L_n)$ with boundary geodesic of lengths $L_1,...,L_n$ and $\vec{o}=(0,0,...,0)$ representing the Neveu-Schwarz boundary. That is, $$\hat{V}_{g,n}(L_1,...,L_n)=\int_{\mathcal{M}_{g,n,\vec{o}}^{spin}(L_1,...,L_n)}e(E_{g,n}^{\vee})\mathrm{exp}( \omega^{WP}).$$ 
Stanford-Witten also derived a recursion formula of $\hat{V}_{g,n}(L_1,...,L_n)$, which can be regarded as the super version of Mirzakhani's recursion formula \cite{mirzakhani2007simple}. Stanford-Witten's formula can be equivalently expressed in terms of super intersection numbers $\mathbf{(\Rmnum{3})}$ (cf. Section 2.1) through a new cohomology class $\Theta$ constructed by Norbury \cite{norbury2023new,norbury2020enumerative} who proved that $e(E_{g,n}^{\vee})$ coincides with $\Theta$ class up to a factor $2^{1-g-n}$ after push-forward,
\begin{equation}\label{VTheta}\begin{aligned}\hat{V}_{g,n}(L_1,...,L_n)&=2^{1-g-n}\int_{\overline{\mathcal{M}}_{g,n}}\Theta_{g,n}\mathrm{exp}(2\pi^2\kappa_1+\frac{1}{2}\sum_{i=1}^n L_i^2\psi_i)\\&=:2^{1-g-n} V^{\Theta}_{g,n}(L_1,...,L_n).\end{aligned}\end{equation}

Now we explain the notations in \eqref{VTheta}. $\overline{\mathcal{M}}_{g,n}$ represents the moduli space of stable $n$-pointed genus $g$ complex curves. The morphism that forgets the last marked point is $$\pi: \overline{\mathcal{M}}_{g,n+1}\to\overline{\mathcal{M}}_{g,n},$$ and the gluing maps which glue the last two points are $$\phi_{irr}: \overline{\mathcal{M}}_{g-1,n+2}\to\overline{\mathcal{M}}_{g,n},$$ $$\phi_{h,I}:\overline{\mathcal{M}}_{h,|I|+1}\times\overline{\mathcal{M}}_{g-h,|J|+1}\to\overline{\mathcal{M}}_{g,n},\qquad I\sqcup J=\{1,...,n\}.$$
The canonical sections of $\pi$ is denoted by $\sigma_1,...,\sigma_n$ and the corresponding divisors
in $\overline{\mathcal{M}}_{g,n+1}$ is denoted by $D_1,...,D_n$. Let $\omega_{\pi}$ be the relative dualizing sheaf, the tautological classes on $\overline{\mathcal{M}}_{g,n}$ are
$$\psi_i=c_1(\sigma_i^{\ast}(\omega_{\pi})),$$ $$\kappa_i=\pi_{\ast}(c_1(\omega_{\pi}(\sum D_i))^{i+1}).$$
Define \begin{equation}\label{WP}
[\prod\limits_{i=1}\limits^{n}\tau_{d_i}]_g=\frac{(2\pi^2)^{d_0}}{d_0!}\prod\limits_{i=1}\limits^{n}2^{2d_i}(2d_i+1)!!\int_{\overline{\mathcal{M}}_{g,n}}\kappa_1^{d_0}\prod\limits_{i=1}\limits^{n}\psi_i^{d_i},
\end{equation}
where $\kappa_1=\omega/(2\pi^2)$ is the first Mumford class on $\overline{\mathcal{M}}_{g,n}$ \cite{arbarello1996combinatorial}. According to \cite{mirzakhani2013growth}, the Weil-Petersson volumes $V_{g,n}$ of $\overline{\mathcal{M}}_{g,n}$ equals $[\tau_0^n]_g$.

The class $\Theta_{g,n}$ in \eqref{VTheta} belongs to $\mathit{H}^{4g-4+2n}(\overline{\mathcal{M}}_{g,n},\mathbb{Q})$ for $g\geq 1$, $n\geq 0$, $2g-2+n>0$ and \\
(\rmnum{1}) $\phi_{irr}^{\ast}\Theta_{g,n}=\Theta_{g-1,n+2}$ and $ \phi_{h,\mathbf{I}}^{\ast}\Theta_{g,n}=\pi_1^{\ast}\Theta_{h,|I|+1}\cdot\pi_2^{\ast}\Theta_{g-h,|J|+1}$,\\
(\rmnum{2})$\Theta_{g,n+1}=\psi_{n+1}\cdot\pi^{\ast}\Theta_{g,n}$,\\ (\rmnum{3}) $\Theta_{1,1}=3\psi_1$, \\where $\pi_i$
is projection onto the $i$-th factor of $\overline{\mathcal{M}}_{h,|I|+1}\times\overline{\mathcal{M}}_{g-h,|J|+1}$. Here we require $g\geq 1$ since $\Theta_{0,n}=0$ which is obtained from that $\mathrm{deg}\ \Theta_{0,n}=n-2>n-3=\mathrm{dim}(\overline{\mathcal{M}}_{g,n})$

For $\mathbf{d}=(d_1,...,d_n)$ with $d_i\in\mathbb{Z}_{\geq 0}$ and $|\mathbf{d}|=d_1+\cdots+d_n\leq g-1,$ let $d_0=g-1-|\mathbf{d}|$ and define 
\begin{equation}\label{superint}
[\prod\limits_{i=1}\limits^{n}\tau_{d_i}]^{\Theta}_g:=\frac{\prod_{i=1}^n(2d_i+1)!!2^{2|\mathbf{d}|}(2\pi^2)^{d_0}}{d_0!}\int_{\overline{\mathcal{M}}_{g,n}}\Theta_{g,n}\psi_1^{d_1}\cdots\psi_n^{d_n}\kappa_1^{d_0}.
\end{equation}
We call it $super$ $intersection$ $number$, which differs from \eqref{WP} in a $\Theta$ class. Therefore, \eqref{VTheta} can be represented by super intersection numbers as
\begin{equation}\label{VTheta2}V_{g,n}^\Theta(2L_1,...,2L_n)=\sum\limits_{|\mathbf{d}|\leq g-1}[\tau_{d_1}\cdots\tau_{d_n}]^\Theta_g\frac{L_1^{2d_1}}{(2d_1+1)!}\cdots\frac{L_n^{2d_n}}{(2d_n+1)!}.\end{equation}
In particular, $V_{g,n}^\Theta=V_{g,n}^\Theta(0,...,0)=[\tau_0^n]^\Theta_g$.

\subsection{Large genus asymptotics for intersection numbers} In a remarkable work  \cite{mirzakhani2015towards}, Mirzakhani-Zograf obtained the large asymptotic expansions of \eqref{WP} and $V_{g,n}$. They devised an algorithm to calculate the asymptotic coefficients. A slightly more detailed exposition of their algorithm and some improvements of the polynomial properties of these coefficients can be found in \cite{Huang1}. The work of Mirzakhani-Zograf \cite{mirzakhani2015towards} opened a new research field on random hyperbolic geometry of large genera, see \cite{mirzakhani2013growth,wu2022random,wu2022small,he2022second,nie2023large,parlier2022simple} for some recent progress.

Masur-Veech volume of moduli space of quadratic differentials in the principle case can be expressed in terms of intersection numbers on moduli space of curves \cite{delecroix2021masur,chen2023masur}. Inspired by this connection, the large genus asymptotics of $\psi$-intersection numbers were derived by Aggarwal \cite{aggarwal2021large} and Guo-Yang \cite{guo2024large} via different methods. Moreover, Guo-Yang \cite{guo2024large} gave very precise description of the asymptotic coefficients as polynomials. The special cases of this polynomial properties were studied in \cite{liu2016recursions,liu2014remark}. The integrand of the Masur-Veech volume and $\Theta$ class both are special $\Omega$-classes \cite{giacchetto2023intersection,chen2023masur,norbury2023new}. We hope that their asymptotics may share some similarities.

In \cite{guo2024combinatorics}, Guo-Norbury-Yang-Zagier studied large genus asymptotics of intersection numbers involving $\Theta$, $\psi$ classes and the polynomial properties of the coefficients in the expansions. In \cite{Huang2}, we start from a closed formula of $2$-point correlators and obtain the large genus asymptotics for the $n$-point correlators. The normalizations in \cite{guo2024combinatorics,Huang2} are different.

\subsection{Main results}
At present, large genus asymptotics for super Weil-Petersson volumes are only known for one boundary case which was proposed in \cite{stanford2020jt} and refined in \cite{okuyama2020jt}.
The conjectural formulae (cf. \cite[(2.13) and (2.11a)]{griguolo2024asymptotics}) of large genus asymptotics of super Weil-Petersson volumes with given boundary components satisfying $L_i\ll g$ was recently proposed by Griguolo-Papalini-Russo-Seminara \cite{griguolo2024asymptotics} by applying resurgence theory to analyze the matrix model of JT supergravity, which we will refer to as the $GPRS$ $conjecture$. We reformulated it following the notation \eqref{VTheta} as below.
\begin{conjecture}[GPRS \cite{griguolo2024asymptotics}]\label{Mainconj}
(1). As $g\to\infty$,
\begin{equation}\label{Vconj}
V^{\Theta}_{g,n}(L_1,...,L_n)\sim\frac{\Gamma(2g+n-\frac{5}{2})}{2^{n-\frac{1}{2}}\pi^{\frac{9}{2}-n-2g}}\prod\limits_{i=1}\limits^{n}\frac{\mathrm{sinh}(L_i/4)}{L_i/4}.
\end{equation}

(2). For any given $\mathbf{d}$, as $g\to\infty$, we have
\begin{equation}\label{Coeffconj}
[\prod\limits_{i=1}\limits^{n}\tau_{d_i}]^{\Theta}_g\sim\frac{2^{-n-2|\mathbf{d}|}\pi^{2g+n-\frac{9}{2}}}{\sqrt{g-|\mathbf{d}|-1}}(2g-3+n)!
\end{equation} 
\end{conjecture}

Our proof of GPRS conjecture follows the strategy of Mirzakhani-Zograf \cite{mirzakhani2015towards}. We generalize their algorithm to the super case (cf. Section 3.1) which can explicitly compute the coefficients in the asymptotic expansions of super intersection numbers \eqref{superint} via the recursion formulae $(\mathbf{\Rmnum{1}})$-$(\mathbf{\Rmnum{3}})$ (cf. Section 2.1) obtained in the author, Liu and Xu's previous paper \cite{HLX}. We prove that these super coefficients are also polynomials. Another key ingredient is an extension of a two-sided estimate of Wu-Xue \cite[Lemma 22]{nie2023large} to the super case. As a result, we prove GPRS conjecture up to a universal constant (cf. Theorem \ref{GPRSConj}).

Compared to the original Mirzakhani's recursion formula \cite[P276 $\mathbf{\Rmnum{3}}$]{mirzakhani2013growth}, the coefficients $a_L$ in the super version of Mirzakhani's recursion formula $(\mathbf{\Rmnum{3}})$ have no closed formula, which made the calculation in the super case more complicated. We will give detailed proofs to take care of this. Besides, some facts frequently used in the following proofs are stated in the Appendix A.

\noindent\textbf{Notations.} In this paper, $f_1(g)\asymp f_2(g)$ means that there exists a universal constant $C<\infty$ such that $$\frac{f_2(g)}{C}\leq f_1(g)\leq Cf_2(g).$$
$f_{1}(g)=\mathit{O}\left(f_2(g)\right)$ means there exists a universal constant $C>0$ such that $$f_1(g)\leq Cf_{2}(g),$$
and $f_{1}(g)=\mathit{o}\left(f_2(g)\right)$ means that $$\lim\limits_{g\to\infty}\frac{f_1(g)}{f_2(g)}=0.$$

Now we discuss the main results obtained in this paper:

\begin{theorem}\label{Polythm}We have the following asymptotic expansions as $g\to\infty$:
\begin{itemize}[leftmargin=2em]
\item [(1).] Given $n\geq 1$, $s\geq 0$, and $\mathbf{d}=(d_1,...,d_n)$, there exist $e^1_{n,\mathbf{d}},...,e^s_{n,\mathbf{d}}$ independent of $g$ such that
\begin{equation}\label{Tauasymp}
\frac{4^{|\mathbf{d}|}[\tau_{d_1}\cdots\tau_{d_n}]_g^\Theta}{V_{g,n}^\Theta}=1+\frac{e^1_{n,\mathbf{d}}}{g}+\cdots+\frac{e^s_{n,\mathbf{d}}}{g^s}+\mathit{O}\left(\frac{1}{g^{s+1}}\right),
\end{equation}
where $|\mathbf{d}|=d_1+\cdots+d_n$. Moreover, for any fixed $n$ and $\mathbf{d}$, the coefficient $e_{n,\mathbf{d}}^i$ is a polynomial in $\mathbb{Q}[\pi^{-1}]$ of degree at most $2i$.
\item [(2).] Given $n,s\geq 0$, there exist $a^i_{n},b^i_{n}$, i=1,...,s, independent of $g$ such that
\begin{equation}\label{Nasymp}
\frac{\frac{\pi}{2}(2g-2+n)V^\Theta_{g,n}}{V^\Theta_{g,n+1}}=1+\frac{a^1_{n}}{g}+\cdots+\frac{a^s_{n}}{g^s}+\mathit{O}\left(\frac{1}{g^{s+1}}\right),
\end{equation}
\begin{equation}\label{Gasymp}
\frac{4V^\Theta_{g-1,n+2}}{V^\Theta_{g,n}}=1+\frac{b^1_{n}}{g}+\cdots+\frac{b^s_{n}}{g^s}+\mathit{O}\left(\frac{1}{g^{s+1}}\right),
\end{equation}
Moreover, each $a_n^i$ and $b^i_n$ is a polynomial in $\mathbb{Q}[n,\pi^{-1}]$ of degree $i$ in $n$ and of degree at most $2i$ in $\pi^{-1}$.
\end{itemize}
\end{theorem}

\begin{remark}\label{Coeffs}
The coefficients $e_{n,\mathbf{d}}^i$, $a_n^i$ and $b_n^i$ in the above expansions \eqref{Tauasymp}-\eqref{Gasymp} can be computed via Algorithm (cf. Section 3.1) recursively. Here, we list the first two terms of them (cf. Theorem 3.1).
$$e^1_{n,\mathbf{d}}=-\frac{4|\mathbf{d}|(|\mathbf{d}|+n-3/2)}{\pi^2},$$ and \begin{align*}e^2_{n,\mathbf{d}}&=\frac{1}{\pi^4}\times\Bigg[8|\mathbf{d}|^4+(16n-40)|\mathbf{d}|^3+\Big(8n^2+(2\pi^2-48)n-6\pi^2+62\Big)|\mathbf{d}|^2\\&\qquad+\Big((2\pi^2-12)n^2-(9\pi^2-44)n-38+9\pi^2-\frac{\pi}{4}\Big)|\mathbf{d}|\Bigg]+\frac{n-s}{16\pi^2},\end{align*}
where $s:=\#\{i|d_i=0\}$ denotes the number of zero in $\mathbf{d}$.
$$ a^1_{n}=\frac{8n-8+\pi^2}{4\pi^2},\qquad b^1_{n}=-\frac{4n-2}{\pi^2},$$ and $$a^2_{n}=\frac{1}{\pi^4}\left[\Big(-\frac{3\pi^2}{2}+8 \Big)n^2-\Big(\frac{\pi^4}{8}-5\pi^2+20\Big)n+\frac{17\pi^4}{64}-\frac{\pi^3}{16}-\frac{27\pi^2}{8}+\frac{\pi}{4}+12 \right],$$ $$b^2_n=\frac{1}{\pi^4}\left[(2\pi^2-4)n^2-(7\pi^2-12)n+\frac{13\pi^2}{8}-\frac{\pi}{2}-8\right].$$
\end{remark}
\begin{remark}
In view of \eqref{summation}-\eqref{CV}, $e^i_{n,\mathbf{d}}$ depends on $|\mathbf{d}|$ and the number of $s$ for $s\leq i-2$ in $\mathbf{d}$. Similar phenomenon also appears in the large genus asymptotics of Weil-Petersson volumes \cite{mirzakhani2015towards,Huang1} and $\psi$-intersection numbers \cite[Lemma 3.7]{liu2014remark} and \cite[Conjecture 1]{guo2024large}.
\end{remark}

We obtain the complete expansions of super Weil-Petersson volumes in terms of inverse powers of $g$ as the following.

\begin{theorem}\label{superWP}
There exists a universal constant $0<C<\infty$ such that for any given $s\geq 1$, $n\geq 0$,
$$V_{g,n}^{\Theta}=C\frac{(2g-3+n)!\pi^{2g+n}}{2^n\sqrt{g}}\left(1+\frac{d_n^1}{g}+\cdots+\frac{d_n^s}{g^s}+\mathit{O}\Big(\frac{1}{g^{s+1}}\Big)\right),\ \ \text{as}\ g\to\infty.$$ Each term $d_n^i$ above is a polynomial in $n$ of degree $2i$ with coefficients in $\mathrm{Q}[\pi^{-1}]$ and can be computed recursively. Moreover, the coefficient of $d_n^i$ at $n^{2i}$ equals $\frac{(-1)^i}{i!\pi^{2i}}$ and 
$$d_n^1=-\frac{n}{4}+\frac{19}{32}-\frac{1}{8\pi}-\frac{1}{\pi^2}\left(n^2-3n+\frac{25}{8}\right).$$
\end{theorem}

From a basic estimate of $\frac{4^{|\mathbf{d}|}[\tau_{d_1}\cdots\tau_{d_n}]^\Theta}{V_{g,n}^\Theta}$ (cf. Lemma \ref{e1-est}), we give an estimate of $ \frac{V_{g,n}^\Theta(2\mathbf{L})}{V_{g,n}^\Theta}$ as the following.

\begin{theorem}\label{Bound}
Let $g,n\geq 1$ and $L_1,...,L_n\geq 0$, then there exists a const $c=c(n)>0$ independent of $g,L_{1},...,L_n$ such that
$$\prod\limits_{i=1}\limits^{n}\frac{\mathrm{sinh}(L_i/2)}{L_i/2}\bigg(1-c(n)\frac{\sum\limits_{i=1}\limits^{n}L_i^2}{g}\bigg)\leq\frac{V_{g,n}^\Theta(2L_1,...,2L_n)}{V_{g,n}^\Theta}\leq\prod\limits_{i=1}\limits^{n}\frac{\mathrm{sinh}(L_i/2)}{L_i/2}.$$ 
\end{theorem}
From the estimates above, we prove Conjecture \ref{Mainconj} up to a universal constant.

\begin{theorem}\label{GPRSConj}The following constant $C$ is the same as in Theorem \ref{superWP}.

(1). Let $n\geq 1$ be fixed and $L_1,...,L_n\geq 0$ satisfying $L_i=\mathit{o}(\sqrt{g})$ for $1\leq i\leq n$, then as $g\to\infty$
\begin{equation}\label{Vconj2}V^{\Theta}_{g,n}(L_1,...,L_n)\sim C\cdot\frac{\Gamma(2g+n-\frac{5}{2})}{2^{n-\frac{1}{2}}\pi^{-n-2g}}\prod\limits_{i=1}\limits^{n}\frac{\mathrm{sinh}(L_i/4)}{L_i/4}.\end{equation}

(2) For any $\mathbf{d}$ satisfying $|\mathbf{d}|=\mathit{o}(g)$, as $g\to\infty$, we have
\begin{equation}\label{Coeffconj2}
[\prod\limits_{i=1}\limits^{n}\tau_{d_i}]^{\Theta}_g\sim C\cdot\frac{2^{-n-2|\mathbf{d}|}\pi^{2g+n}(2g-3+n)!}{\sqrt{g-|\mathbf{d}|-1}}.\end{equation}
\end{theorem}

In view of Conjecture \ref{Mainconj} and Theorem \ref{GPRSConj}, the universal constant $C$ should be equal to $\pi^{-\frac{9}{2}}$.

\begin{conjecture}
$\displaystyle C=\pi^{-\frac{9}{2}}$ in Theorem \ref{superWP}.
\end{conjecture}

When $n$ varies as $g\to\infty$, we get
\begin{theorem}\label{NsuperWP}
For any $n\geq 0$ satisfying $n=\mathit{o}(\sqrt{g})$, then as $g\to\infty$, $$V_{g,n(g)}^\Theta=C\frac{(2g-3+n(g))!\pi^{2g+n}}{2^n\sqrt{g}}\left(1+\mathit{O}\Big(\frac{n(g)^2}{g}\Big)\right),$$
where $C$ is the same constant as in Theorem \ref{superWP}.\end{theorem}

\textbf{Outlines.} We will give some inequalities of super intersection numbers as well as super Weil-Petersson volumes which are obtained from recursion formulae $(\mathbf{\Rmnum{1}})$-$(\mathbf{\Rmnum{3}})$ in Section 2. As a consequence, we prove Theorem \ref{Bound}. In Section 3, we will give a detailed algorithm which can calculate all the coefficients in \eqref{Tauasymp}-\eqref{Gasymp} and then use it to compute the first three terms stated in Remark \ref{Coeffs}.  As a result, we will prove the polynomial properties of these cofficients in Section 4 to complete the proof of Theorem \ref{Polythm}. In Section 5, we will prove Theorem \ref{superWP} and \ref{GPRSConj}. In the last section, we will give a proof of Theorem \ref{NsuperWP}.

\

\noindent{\bf Acknowledgements} We thank Kefeng Liu and Hao Xu for very helpful discussions.

\vskip 20pt
\section{Basic estimates}\setcounter{equation}{0}

\subsection{Recursion formulae} In this paper, super intersection numbers \eqref{superint} are calculated via the following recursion formulae with initial data:
$$[\tau_0]_{1}^\Theta=\frac{1}{8}\qquad \text{and}\qquad[\ \ ]_{2}^\Theta=2\pi^2\int_{\overline{\mathcal{M}}_{2}}\Theta_2\kappa_1=\frac{3\pi^2}{64}.$$
Note that we require $g\geq 1$ since $\Theta_{0,n}=0$.

Given $\mathbf{d}=(d_1,...,d_n)$ with $|\mathbf{d}|\leq g-1,$ the following recursion formulae hold:
\begin{itemize}[leftmargin=2em]
\item [$\mathbf{(\Rmnum{1})}$.] $$[\tau_0\tau_1\prod\limits_{i=1}\limits^{n}\tau_{d_i}]^{\Theta}_g=[\tau_0^4\prod\limits_{i=1}\limits^{n}\tau_{d_i}]^{\Theta}_{g-1}+6\sum_{\substack{g_1+g_2=g\\[3pt]I\sqcup J=\{1,...,n\}}}[\tau_0^2\prod\limits_{i\in I}\tau_{d_i}]^{\Theta}_{g_1}[\tau_0^2\prod\limits_{i\in J}\tau_{d_i}]^{\Theta}_{g_2}.$$
\item [$\mathbf{(\Rmnum{2})}$.] $$(2g-2+n)[\prod\limits_{i=1}\limits^{n}\tau_{d_i}]^{\Theta}_g=\sum\limits_{L=0}\limits^{g-1}(-1)^L\frac{\pi^{2L}}{(2L+1)!}[\tau_L\prod\limits_{i=1}\limits^{n}\tau_{d_i}]^{\Theta}_g.$$
\item [$\mathbf{(\Rmnum{3})}$.] Let $a_{-1}=0$, $a_0=1$ and for $n\geq 1$, $$\frac{1}{cos(\pi x)}=\sum\limits_{n=0}^{\infty}a_nx^{2n},\ \ \text{or equivalently}\ \ a_n=\frac{1}{\pi}\int_{0}^{\infty}\frac{1}{\mathrm{cosh}(x/2)}\cdot\frac{x^{2n}}{(2n)!}dx.$$ Then we have
$$[\prod\limits_{i=1}\limits^{n}\tau_{d_i}]^{\Theta}_g=\sum\limits_{j=2}^n A^j_{\mathbf{d}}+B_{\mathbf{d}}+C_{\mathbf{d}},$$
where \begin{equation}\label{Aj}A^j_{\mathbf{d}}=\sum\limits_{L=0}^{d_0}(2d_j+1)a_L[\tau_{L+d_1+d_j}\prod\limits_{i\neq 1,j}\tau_{d_i}]^{\Theta}_g,\end{equation}
\begin{equation}\label{B}B_{\mathbf{d}}=2\sum\limits_{L=0}^{d_0}\sum\limits_{k_1+k_2=L+d_1-1}a_L[\tau_{k_1}\tau_{k_2}\prod\limits_{i\neq 1}\tau_{d_i}]_{g-1}^\Theta,\end{equation}
and
\begin{equation}\label{C}C_{\mathbf{d}}=2\sum_{\substack{g_1+g_2=g\\[3pt]I\sqcup J=\{2,...,n\}}}\sum\limits_{L=0}^{d_0}\sum\limits_{k_1+k_2=L+d_1-1}a_L[\tau_{k_1}\prod\limits_{i\in I}\tau_{d_i}]_{g_1}^\Theta[\tau_{k_2}\prod\limits_{i\in J}\tau_{d_i}]_{g_2}^\Theta.\end{equation}
\end{itemize}
\begin{remark}
(1). $\mathbf{(\Rmnum{1})}$ and $\mathbf{(\Rmnum{2})}$ are special cases of \cite[Proposition 4.1 and Proposition 4.2]{HLX}, respectively. $\mathbf{(\Rmnum{3})}$ is the super version of Mirzakhani's recursion formula, which is equivalent to Stanford-Witten's recursion formula \cite[D.30]{stanford2020jt}. See \cite[Theorem 2]{norbury2020enumerative} and \cite[Theorem A.4]{HLX}. \\
(2). $\mathbf{(\Rmnum{2})}$ was also obtaied by Norbury \cite[Theorem 6.14]{norbury2020enumerative}  as the following. $$V_{g,n+1}^\Theta (2\pi i,L_1,...L_n)=(2g-2+n)V_{g,n}^\Theta (L_1,...L_n).$$
(3). Some explict formulae of super intersection numbers were obtained via BGW tau-function in \cite{dubrovin2021tau,yang2021gelfand,yang2021extension}.
\end{remark}

\subsection{Numerical properties of $\mathbf{a_n}$} Note that in Recursion $\mathbf{(\Rmnum{3})}$, $$a_n=\frac{1}{\pi}\int_{0}^{\infty}\frac{1}{\mathrm{cosh}(x/2)}\cdot\frac{x^{2n}}{(2n)!}dx.$$ Theofore, we have \begin{proposition}\begin{equation}\label{an}\frac{a_{n}}{4^{n}}-\frac{a_{n-1}}{4^{n-1}}=\frac{1}{\pi}\int_{0}^\infty \frac{e^{-x/2}}{(\mathrm{cosh} \frac{x}{2})^2}\bigg(\frac{(x/2)^{2n}}{(2n)!}+\frac{(x/2)^{2n-1}}{(2n-1)!}\bigg)dx \end{equation}\end{proposition}

\begin{proof}
\begin{align*}\frac{a_{n}}{4^{n}}-RHS&=\frac{1}{\pi}\bigg[\int_{0}^\infty \frac{\mathrm{sinh}\frac{x}{2}}{(\mathrm{cosh} \frac{x}{2})^2}\cdot\frac{(x/2)^{2n}}{(2n)!}dx-\int_{0}^\infty\frac{e^{-x/2}}{(\mathrm{cosh} \frac{x}{2})^2}\cdot\frac{(x/2)^{2n-1}}{(2n-1)!}dx \bigg]\\&=\frac{1}{\pi}\bigg[\int_{0}^\infty \frac{1}{\mathrm{cosh} \frac{x}{2}}\cdot\frac{(x/2)^{2n-1}}{(2n-1)!}dx-\int_{0}^\infty\frac{e^{-x/2}}{(\mathrm{cosh} \frac{x}{2})^2}\cdot\frac{(x/2)^{2n-1}}{(2n-1)!}dx \bigg]\\&=\frac{1}{\pi}\int_{0}^\infty\frac{\mathrm{sinh}\frac{x}{2}}{(\mathrm{cosh} \frac{x}{2})^2}\frac{x^{2n-1}}{(2n-1)!}dx\\&=\frac{1}{\pi}\int_{0}^\infty\frac{1}{\mathrm{cosh} \frac{x}{2}}\frac{x^{2n-2}}{(2n-2)!}dx=\frac{a_{n-1}}{4^{n-1}}.\end{align*}
\end{proof}

\begin{lemma}\label{Coeff-est}
The sequence $\{a_i\}_{i=-1}^\infty$ with $a_{-1}=0$ and $a_0=1$ is increasing. Moreover,
\begin{itemize}[leftmargin=2em]
\item [(1)]$\displaystyle \frac{a_{i}}{4^{i}}-\frac{a_{i-1}}{4^{i-1}}\asymp \frac{1}{4^i}$.\\
\item [(2)]$\displaystyle\sum\limits_{i=0}\limits^\infty \bigg(\frac{a_{i}}{4^{i}}-\frac{a_{i-1}}{4^{i-1}}\bigg)=\frac{4}{\pi}$.\\
\item [(3)]$\displaystyle\sum\limits_{i=0}\limits^\infty i \bigg(\frac{a_{i}}{4^{i}}-\frac{a_{i-1}}{4^{i-1}}\bigg)=\frac{1}{\pi}$.\\
\item [(4)] For $j\in\mathbb{Z}$ and $j\geq 2$, the sum $$\pi\sum\limits_{i=0}\limits^\infty i^j \bigg(\frac{a_{i}}{4^{i}}-\frac{a_{i-1}}{4^{i-1}}\bigg)$$ is a polynomial in $\pi^2$ of degree $\lfloor j/2 \rfloor$ with rational coefficients.
\end{itemize}
\end{lemma}
\begin{proof}
Part (1) and (2) can be easily checked from \eqref{an}. As for part (3), a simple calculation shows \begin{align*}2\sum\limits_{i=0}\limits^\infty i \bigg(\frac{a_{i}}{4^{i}}-\frac{a_{i-1}}{4^{i-1}}\bigg)&=2\sum\limits_{i=0}\limits^\infty(i+1) \bigg(\frac{a_{i+1}}{4^{i+1}}-\frac{a_i}{4^i}\bigg)\\&=\frac{1}{\pi}\int_{0}^\infty \frac{e^{-x/2}}{(\mathrm{cosh} \frac{x}{2})^2}\sum\limits_{i=0}^\infty\bigg(\frac{(x/2)^{2i+2}}{(2i+1)!}+\frac{(x/2)^{2i+1}}{(2i)!}+\frac{(x/2)^{2i+1}}{(2i+1)!}\bigg)dx\\&=\frac{1}{\pi}\int_{0}^\infty \frac{e^{-x/2}}{(\mathrm{cosh} \frac{x}{2})^2}\bigg(\frac{x}{2}e^{x/2}+\mathrm{sinh}\frac{x}{2}\bigg)dx\\&=\frac{2}{\pi}.\end{align*}

Also, it suffices to prove $$\pi\sum\limits_{i=0}\limits^\infty (i+1)^j \bigg(\frac{a_{i+1}}{4^{i+1}}-\frac{a_{i}}{4^{i}}\bigg)$$ is a polynomial in $\pi^2$ of degree $\lfloor j/2 \rfloor$. From \eqref{an}, one has
$$\pi\bigg(\frac{a_{i+1}}{4^{i+1}}-\frac{a_i}{4^i}\bigg)=\int_{0}^\infty \frac{e^{-x/2}}{(\mathrm{cosh} \frac{x}{2})^2}\bigg(\frac{(x/2)^{2i+2}}{(2i+2)!}+\frac{(x/2)^{2i+1}}{(2i+1)!}\bigg)dx.$$
Denote $\displaystyle D:=\frac{x}{2}\frac{d}{dx}$ and $\displaystyle F:=\frac{d}{dx}\frac{x}{2}$, then $$\sum\limits_{i=0}\limits^\infty(i+1)^j\frac{(x/2)^{2i+2}}{(2i+2)!}=D^j(\mathrm{cosh}\frac{x}{2}-1)=:\sum\limits_{l=0}\limits^{j}\bigg(\frac{x}{2}\bigg)^l\bigg(m_{1,j}^{(l)}\mathrm{cosh}\frac{x}{2}+n_{1,j}^{(l)}\mathrm{sinh}\frac{x}{2}\bigg),$$ and $$\sum\limits_{i=0}\limits^\infty(i+1)^j\frac{(x/2)^{2i+1}}{(2i+1)!}=F^j(\mathrm{sinh}\frac{x}{2})=:\sum\limits_{l=0}\limits^{j}\bigg(\frac{x}{2}\bigg)^l\bigg(m_{2,j}^{(l)}\mathrm{cosh}\frac{x}{2}+n_{2,j}^{(l)}\mathrm{sinh}\frac{x}{2}\bigg),$$
where $\{m_{\ast,j}^{(l)}\}_{l=0}^{j}$ and $\{n_{\ast,j}^{(l)}\}_{l=0}^{j}$ are rational numbers and they can be computed recursively from $j$ to $j+1$. Recall that $\displaystyle\zeta(l+1)=\frac{1}{l!}\int_{0}^\infty \frac{x^l}{e^x-1}dx.$ Theofore, we have \begin{align*}&\int_{0}^\infty \frac{e^{-x/2}}{(\mathrm{cosh} \frac{x}{2})^2}\cdot\mathrm{cosh} \frac{x}{2}\cdot\bigg(\frac{x}{2}\bigg)^l dx\\&=2^{1-l}\int_0^\infty\frac{x^l}{e^x+1}dx \\&=2^{1-l}\int_0^\infty\bigg(\frac{1}{e^x+1}-\frac{1}{e^x-1}\bigg)x^ldx+2^{1-l}\int_0^\infty\frac{x^l}{e^x-1}dx\\&=-2^{2-l}\int_0^\infty\frac{x^l}{e^{2x}-1}dx+l!2^{1-l}\zeta(l+1)\\&=l!(2^{1-l}-2^{1-2l})\zeta(l+1).\end{align*}
Similarly, $$\int_{0}^\infty \frac{e^{-x/2}}{(\mathrm{cosh} \frac{x}{2})^2}\cdot\mathrm{sinh} \frac{x}{2}\cdot\bigg(\frac{x}{2}\bigg)^l dx=l![(2^{2-l}-2^{3-2l})\zeta(l)-(2^{1-l}-2^{1-2l})\zeta(l+1)]. $$
Therefore, the two identities above implies that $\zeta(l)$ for $l=2k+1$ and $0\leq l\leq j$ makes no contribution to the sum $$\pi\sum\limits_{i=0}\limits^\infty(i+1)^j \bigg(\frac{a_{i+1}}{4^{i+1}}-\frac{a_i}{4^i}\bigg)$$ if and only if $$(2k+1)\bigg(m_{2,j}^{(2k+1)}+n_{1,j}^{(2k+1)}\bigg)+m_{1,j}^{(2k)}-m_{2,j}^{(2k)}-n_{1,j}^{(2k)}+n_{2,j}^{(2k)}=0,$$ which is easy to verify by induction from $j$ to $j+1$. Thus only $\zeta(2k)$ for $1\leq k\leq\lfloor\frac{j}{2} \rfloor$ and rational numbers $m_{\ast,j}^{(l)}$, $n_{\ast,j}^{(l)}$ for $0\leq l\leq j$ contribute to the summation $\displaystyle \pi\sum\limits_{i=0}\limits^\infty(i+1)^j \bigg(\frac{a_{i+1}}{4^{i+1}}-\frac{a_i}{4^i}\bigg)$, which implies part (4).
\end{proof}

\subsection{Basic properties for super intersection numbers} 

\begin{lemma}\label{Basic-est}
For any $g\geq 1$ and $n\geq 0$,
\begin{itemize}[leftmargin=2em]
\item [(1).]$\displaystyle 4^{|\mathbf{d}|+1}[\tau_{d_1+1}\prod\limits_{i=2}\limits^{n}\tau_{d_i}]^\Theta_g\leq4^{|\mathbf{d}|}[\prod\limits_{i=1}\limits^{n}\tau_{d_i}]^\Theta_g\leq V_{g,n}^\Theta$\\
\item [(2).] \eqref{Gre} and \eqref{Nre} imply the following result respectively: \begin{equation}\label{G1}4V_{g-1,n+4}^\Theta\leq V_{g,n+2}^\Theta,\end{equation} and \begin{equation}\label{N1}b_0<\frac{\frac{\pi}{2}(2g-2+n)V_{g,n}^\Theta}{ V_{g,n+1}^\Theta}<b_1,\end{equation} where $\displaystyle b_1=\sum\limits_{L=0}\limits^{\infty}\frac{(\pi/2)^{2L+1}}{(2L+1)!}=\mathrm{sinh}(\pi/2)$ and $\displaystyle b_0=\frac{\pi}{2}-\frac{\pi^3}{48}.$
\end{itemize}
\end{lemma}
\begin{proof}
Set $\mathbf{d}=(d_1,d_2,...,d_n)$ and $\mathbf{d}=(d_1+1,d_2,...,d_n)$ in formula $(\mathbf{\Rmnum{3}})$, respectively, with the help of \eqref{an} and non-negativity of super intersection numbers, we have
$$[\tau_{d_1}\prod\limits_{i=2}\limits^{n}\tau_{d_i}]_g^\Theta-4[\tau_{d_1+1}\prod\limits_{i=2}\limits^{n}\tau_{d_i}]_g^\Theta\geq 0.$$
Hence, (1) follows. 

By part (1) and formula $(\mathbf{\Rmnum{1}})$, we have
$$\frac{4V_{g-1,n+4}^\Theta}{V_{g,n+2}^\Theta}< \frac{4[\tau_1\tau_0^{n+1}]_g^\Theta}{V_{g,n+2}^\Theta}\leq 1.$$
 
From part (1) and formula $(\mathbf{\Rmnum{2}})$, we get
\begin{align*}
\frac{\frac{\pi}{2}(2g-2+n)V_{g,n}^\Theta}{ V_{g,n+1}^\Theta}\leq\sum\limits_{L=0}\limits^{g-1}\frac{(\pi/2)^{2L+1}}{(2L+1)!}=\mathrm{sinh}\left(\frac{\pi}{2}\right),
\end{align*}
and 
$$\frac{\frac{\pi}{2}(2g-2+n)V_{g,n}^\Theta}{ V_{g,n+1}^\Theta}\geq \frac{\pi}{2}-\frac{\pi^3}{48}\frac{[\tau_1\tau_0^n]_g^\Theta}{ V_{g,n+1}^\Theta}\geq \frac{\pi}{2}-\frac{\pi^3}{48}.$$
Therefore, part (2) is proved.
\end{proof}

\begin{remark} For any $g\geq 1$ and $n\geq 0$, \eqref{G1} and \eqref{N1} imply
$$V_{g-1,n+4}^\Theta=\mathit{O}(V_{g,n}^\Theta),$$
and 
\begin{equation}\label{N11}
V_{g,n}^\Theta=\mathit{O}\left(\frac{V_{g,n+1}^\Theta}{g}\right).
\end{equation}
Moreover, $V_{g-1,3}^\Theta=\mathit{O}(V_{g,1}^\Theta)$ can be verified by $$\frac{V_{g-1,3}^\Theta}{V_{g,1}^\Theta}=\frac{V_{g-1,3}^\Theta}{V_{g-1,4}^\Theta}\cdot\frac{V_{g-1,4}^\Theta}{V_{g,2}^\Theta}\cdot\frac{V_{g,2}^\Theta}{V_{g,1}^\Theta},$$
and so is $V_{g-1,2}^\Theta=\mathit{O}(V_{g,0}^\Theta)$. Therefore, we have
\begin{equation}\label{G11}
V_{g-1,n+2}^\Theta=\mathit{O}(V_{g,n}^\Theta).
\end{equation}
\end{remark}

Now we obtain the relation between general super intersection number and super Weil-Petersson volume by \eqref{N11} and \eqref{G11}.

\begin{lemma}\label{e1-est}For any given $\mathbf{d}=(d_1,...,d_n)$, there exists a constant $c_0=c(n)>0$ such that
$$ 0\leq 1-\frac{4^{|\mathbf{d}|}[\tau_{d_1}\cdots\tau_{d_n}]^\Theta_g}{V_{g,n}^\Theta}\leq c_0\frac{|\mathbf{d}|^2}{g}.$$
\end{lemma}
\begin{proof}
The lower bound is obvious by Lemma \ref{Basic-est} (1). By \eqref{N11} and \eqref{G11}, formula $(\mathbf{\Rmnum{1}})$ implies
\begin{equation}\label{Times1}
\sum_{\substack{g_1+g_2=g\\[3pt]I\sqcup J=\{2,...,n\}}}V_{g_1,|I|+1}^\Theta \cdot V_{g_2,|J|+1}^\Theta=\mathit{O}\left(\frac{V_{g,n}^\Theta}{g}\right),
\end{equation}
where the implied constant is independent of $g$.

Therefore, we get the following from formula $(\mathbf{\Rmnum{1}})$ with the help of Lemma \ref{Coeff-est} (2)-(3), Lemma \ref{Basic-est} and \eqref{N11}-\eqref{Times1}, 
\begin{align*}
&4^{|\mathbf{d}|}[\tau_{d_1}\prod\limits_{i=2}\limits^{n}\tau_{d_i}]_g^\Theta-4^{|\mathbf{d}|+1}[\tau_{d_1+1}\prod\limits_{i=2}\limits^{n}\tau_{d_i}]_g^\Theta\\&\leq \sum\limits_{j=2}\limits^{n}(2d_j+1)\sum\limits_{L=0}\limits^{g-1-|\mathbf{d}|}\Big(\frac{a_{L}}{4^{L}}-\frac{a_{L-1}}{4^{L-1}}\Big)V_{g,n-1}^\Theta+8\sum\limits_{L=0}\limits^{g-1-|\mathbf{d}|}(d_1+L)\Big(\frac{a_{L}}{4^{L}}-\frac{a_{L-1}}{4^{L-1}}\Big) V_{g-1,n+1}^\Theta\\&\qquad+8\sum\limits_{L=0}\limits^{g-1-|\mathbf{d}|}(d_1+L)\Big(\frac{a_{L+1}}{4^{L}}-\frac{a_{L-1}}{4^{L-1}}\Big)\sum_{\substack{g_1+g_2=g\\[3pt]I\sqcup J=\{2,...,n\}}}V_{g_1,|I|+1}^\Theta \cdot V_{g_2,|J|+1}^\Theta\\&\leq\frac{8|\mathbf{d}|-8d_1+4n-4}{\pi}V_{g,n-1}^\Theta+\frac{32d_1+8}{\pi}V_{g-1,n+1}^\Theta+\frac{32d_1+8}{\pi}\\&\qquad\times\sum_{\substack{g_1+g_2=g\\[3pt]I\sqcup J=\{2,...,n\}}}V_{g_1,|I|+1}^\Theta \cdot V_{g_2,|J|+1}^\Theta\\&\leq c(n)|\mathbf{d}|\frac{V_{g,n}^\Theta}{g}.
\end{align*} 
Then the upper bound is a consequnce from \eqref{summation}. So we conclude this lemma.
\end{proof}

From Lemma \ref{e1-est}, we obtain the relation between $V_{g,n}^\Theta$ and $V_{g,n}^\Theta(L_1,...,L_n)$. The proof follows the strategy used in the proof of \cite[Lemma 22]{nie2023large} by Nie-Wu-Xue.

\noindent\textbf{Proof of Theorem \ref{Bound}.} Since
$$\frac{V_{g,n}^\Theta(2L_1,...,2L_n)}{V_{g,n}^\Theta}=\sum\limits_{|\mathbf{d}|\leq g-1}\frac{4^{|\mathbf{d}|}[\tau_{d_1}\cdots\tau_{d_n}]^\Theta_g}{V_{g,n}^\Theta}\cdot\frac{(L_1/2)^{2d_1}}{(2d_1+1)!}\cdots\frac{(L_n/2)^{2d_n}}{(2d_n+1)!},$$
then the upper bound is a simple consequence from Lemma \ref{e1-est}.

As for the lower bound, first, note that 
\begin{align*}\frac{L_1^2}{4}\cdot\prod\limits_{i=1}\limits^{n}\frac{\mathrm{sinh}(L_i/2)}{L_i/2}&=\left(\sum\limits_{d_1=1}\limits^{\infty}\frac{(L_1/2)^{2d_1}}{(2d_1-1)!}\right)\cdot\Bigg(\sum\limits_{d_2,...,d_n=0}\limits^{\infty}\frac{(L_2/2)^{2d_2}\cdots(L_n/2)^{2d_n}}{(2d_2+1)!\cdots(2d_n+1)!}\Bigg)\\&=\sum\limits_{d_1,...,d_n=0}\limits^{\infty}2d_1(2d_1+1)\frac{(L_1/2)^{2d_1}\cdots(L_n/2)^{2d_n}}{(2d_1+1)!\cdots(2d_n+1)!}.\end{align*}
Then, $$\Bigg(\frac{\sum_{i=1}^{n}L_i^2}{4}\Bigg)\cdot\prod\limits_{i=1}\limits^{n}\frac{\mathrm{sinh}(L_i/2)}{L_i/2}=\sum\limits_{d_1,...,d_n=0}\limits^{\infty}\frac{(L_1/2)^{2d_1}\cdots(L_n/2)^{2d_n}}{(2d_1+1)!\cdots(2d_n+1)!}\sum\limits_{i=1}\limits^{n}2d_i(2d_i+1).$$
Therefore, by Cauchy-Schwarz inequality
\begin{equation}\label{1.81}\sum\limits_{|\mathbf{d}|\leq g-1}|\mathbf{d}|^2\frac{(L_1/2)^{2d_1}}{(2d_1+1)!}\cdots\frac{(L_n/2)^{2d_n}}{(2d_n+1)!}\leq\frac{n}{16}\Bigg(\sum_{i=1}^{n}L_i^2\Bigg)\cdot\prod\limits_{i=1}\limits^{n}\frac{\mathrm{sinh}(L_i/2)}{L_i/2}.\end{equation}
The Stirling's formula implies that for large $k>0$, one has 
$$k!\geq \left(\frac{k}{e}\right)^k.$$
 Hence, we have
\begin{align*}
\sum\limits_{|\mathbf{d}|> g-1}\frac{(L_1/2)^{2d_1}}{(2d_1+1)!}\cdots\frac{(L_n/2)^{2d_n}}{(2d_n+1)!}&\leq \sum\limits_{k>g-1}\frac{1}{k!}\sum\limits_{|\mathbf{d}|=k}\frac{k!}{d_1!\cdots d_n!}\bigg(\frac{L_1^2}{4}\bigg)^{d_1}\cdots\bigg(\frac{L_n^2}{4}\bigg)^{d_n}\\&\qquad=\sum\limits_{k>g-1}\frac{1}{k!}\left(\frac{L_1^2+\cdots+L_n^2}{4}\right)^k\\&\leq\sum\limits_{k>g-1}\left(\frac{e(L_1^2+\cdots+L_n^2)}{4k}\right)^k.
\end{align*}
On the one hand, when $\displaystyle \frac{e(L_1^2+\cdots+L_n^2)}{4g}\leq\frac{1}{2}$, we get the following from above
\begin{align*}\sum\limits_{|\mathbf{d}|> g-1}\frac{(L_1/2)^{2d_1}}{(2d_1+1)!}\cdots\frac{(L_n/2)^{2d_n}}{(2d_n+1)!}&\leq 2\left(\frac{e(L_1^2+\cdots+L_n^2)}{4g}\right)^g\\&\leq\frac{4(L_1^2+\cdots+L_n^2)}{g}\prod\limits_{i=1}\limits^{n}\frac{\mathrm{sinh}(L_i/2)}{L_i/2}.\end{align*}
With the help of Lemma \ref{e1-est} and \eqref{1.81}, we obtain that there exists a constant $c_0=c(n)>0$ such that
\begin{align*}
\frac{V_{g,n}^\Theta(2L_1,...,2L_n)}{V_{g,n}^\Theta}&\geq\prod\limits_{i=1}\limits^{n}\frac{\mathrm{sinh}(L_i/2)}{L_i/2}-\sum\limits_{|\mathbf{d}|> g-1}\frac{(L_1/2)^{2d_1}}{(2d_1+1)!}\cdots\frac{(L_n/2)^{2d_n}}{(2d_n+1)!}\\&\qquad-\frac{c_0}{g}\sum\limits_{|\mathbf{d}|\leq g-1}|\mathbf{d}|^2\frac{(L_1/2)^{2d_1}}{(2d_1+1)!}\cdots\frac{(L_n/2)^{2d_n}}{(2d_n+1)!}\\&\geq\prod\limits_{i=1}\limits^{n}\frac{\mathrm{sinh}(L_i/2)}{L_i/2}\Bigg(1-\frac{(4+nc_0/16)\sum\limits_{i=1}\limits^{n}L_i^2}{g}\Bigg).
\end{align*} 
On the other hand, if $\displaystyle \frac{e(L_1^2+\cdots+L_n^2)}{4g}>\frac{1}{2}$, the lower bound is trivial.
\qed

\subsection{Estimates for the join part} The technique to evaluate the third term in $(\mathbf{\Rmnum{3}})$ as the following is to use induction in $(\mathbf{\Rmnum{3}})$.

\begin{lemma}\label{VproductBasic}
Given $r\geq 1$, $n\geq 0$ and $k\geq 0$, there exists $C=C(r,n,k)>0$ such that for any $(g_1,m)$, $(g_2,n)$ with $2g_1+m\geq r$ and $|\mathbf{d}|=d_1+\cdots+d_n=k$, we have
\begin{equation}\label{Key}
[\tau_{d_1}\cdots\tau_{d_m}]_{g_1}^\Theta\times V^\Theta_{g_2,n+r}\leq C(r,n,k)\times [\tau_{d_1}\cdots\tau_{d_m}\tau_0^{n}]_{g_1+g_2,m+n}^\Theta
\end{equation}
\end{lemma}

\begin{proof}
From $\mathbf{(\Rmnum{3})}$, we deduce that if $g\leq g^{'}$ and $n\leq n^{'}$, $$[\prod\limits_{i=1}\limits^{n}\tau_{d_i}]_{g}^\Theta\leq[\prod\limits_{i=1}\limits^{n}\tau_{d_i}\tau_0^{n^{'}-n}]_{g^{'}}^\Theta.$$
Then we consider the following two cases.

\noindent\textbf{(Case 1)} If $r\leq 2g_1+m\leq 3r$, then by \eqref{N11} and \eqref{G11}, we have $$\frac{V_{g_2,n+r}}{V_{g_1+g_2,n+m}}=\mathit{O}(1).$$
So by Lemma \ref{e1-est}, there exists two constants $C_{0}(k,n)>0$ and $C(r,n)>0$ such that for any $(g_1,m)$, $(g_2,n)$ with $r\leq 2g_1+m\leq 3r$ and $g_2\geq C(r,n)$, we have
$$[\tau_{d_1}\cdots\tau_{d_m}]_{g_1}^\Theta\times V_{g_2,n+r}=\mathit{O}(V_{g_1+g_2,m+n})\leq C_{0}(k,n)[\tau_{d_1}\cdots\tau_{d_m}\tau_0^{n}]_{g_1+g_2,m+n}^\Theta.$$
Hence, let $$C(r,n,k)=C_{0}(k,n)+\mathrm{max}\{V_{g_2,n+r}\}_{g_2\leq C(r,n)}.$$

\noindent\textbf{(Case 2)} When $2g_1+m> 3r$, we use induction on $2g_1+m$. Suppose the lemma holds for any $(g,n)$ with $2g+n<2g_1+m$ and $2g_1+m> 3r$. Then we use $\mathbf{(\Rmnum{3})}$ to expand both sides of \eqref{Key} and compare them with each other. We write the expansion of both sides as below
$$[\prod\limits_{i=1}\limits^{m}\tau_{d_i}]_{g_1}^\Theta=\sum\limits_{j=2}^m A^j_{\mathbf{d}}+B_{\mathbf{d}}+C_{\mathbf{d}},$$
and $$[\prod\limits_{i=1}\limits^{m}\tau_{d_i}\tau_0^n]_{g_1+g_2}^\Theta=\sum\limits_{j=2}^m A^j_{\mathbf{d^{'}}}+B_{\mathbf{d^{'}}}+C_{\mathbf{d^{'}}},$$
where ${\mathbf{d^{'}}}=(d_1,...,d_m,0,...,0).$ Now we analyze each term of the right hand side of two identities above.

\noindent$\bullet$ \textbf{Contribution from} $\mathbf{A_d}^j$. For $2\leq m$, each term in $A^j_{\mathbf{d}}$ is of the form $a_l\cdot[\tau_{\mathbf{d}(l,j)}]_{g_1}^\Theta$ for $0\leq l\leq g-1-|\mathbf{d}|$, where $\mathbf{d}(l,j)$=$(d_1+d_j+l,d_2,...,\widehat{d_j},...,d_n)$. In this case, the induction hypothesis for $[\tau_{\mathbf{d}(l,j)}]_{g_1}^\Theta$ implies that
$$a_l\cdot[\tau_{\mathbf{d}(l,j)}]_{g_1}^\Theta\times V_{g_2,n+r}^\Theta\leq C(r,n,k)a_l\cdot[\tau_{\mathbf{d}(l,j)}\tau_0^n]_{g_1+g_2}^\Theta.$$
Moreover, all terms of $a_l\cdot[\tau_{\mathbf{d}(l,j)}\tau_0^n]_{g_1+g_2}^\Theta$ appears in the expansion of $A^j_{\mathbf{d^{'}}}$.

\noindent$\bullet$ \textbf{Contribution from} $\mathbf{B_d}$. Similarly, each term in $B_{\mathbf{d}}$ takes the form $a_l\cdot[\tau_{\mathbf{d}(k_1,k_2)}]_{g_1}^\Theta$ for $0\leq l\leq g-1-|\mathbf{d}|$, where $\mathbf{d}(k_1,k_2)$=$(k_1,k_2,d_2,...,d_n)$. The induction hypothesis for $[\tau_{\mathbf{d}(k_1,k_2)}]_{g_1}^\Theta$ implies that 
$$a_l\cdot[\tau_{\mathbf{d}(k_1,k_2)}]_{g_1}^\Theta\times V_{g_2,n+r}^\Theta\leq C(r,n,k)a_l\cdot[\tau_{\mathbf{d}(k_1,k_2)}\tau_0^n]_{g_1+g_2}^\Theta.$$
In this case, all terms of $a_l\cdot[\tau_{\mathbf{d}(k_1,k_2)}\tau_0^n]_{g_1+g_2}^\Theta$ appears in the expansion of $B_{\mathbf{d^{'}}}$.

\noindent$\bullet$ \textbf{Contribution from} $\mathbf{C_d}$. Finally, each term in $C_{\mathbf{d}}$ is of the form $a_l\cdot[\tau_\mathbf{d_1}]^\Theta_{g^{'}}\cdot[\tau_\mathbf{d_2}]_{g_1-g^{'}}^\Theta$ for $0\leq l\leq g-1-|\mathbf{d}|$, where $\mathbf{d_1}=(k_1,\mathbf{d}(I))$ and $\mathbf{d_2}=(k_2,\mathbf{d}(J))$ with $k_1+k_2=l+d_1-1$ as well as $I\sqcup J=\{2,...,n\}$. In this case, we apply the induction hypothesis for $[\tau_\mathbf{d_1}]^\Theta_{g^{'}}$ since $r<2g^{'}+|I|+1<2g_1+m$. Then we have 
$$a_l\cdot[\tau_\mathbf{d_1}]^\Theta_{g^{'}}\cdot[\tau_\mathbf{d_2}]_{g_1-g^{'}}^\Theta\times V_{g_2,n+r}^\Theta\leq C(r,n,k)a_l\cdot[\tau_\mathbf{d_1}\tau_0^n]_{g^{'}+g_2}^\Theta\cdot[\tau_\mathbf{d_2}]_{g_1-g^{'}}^\Theta.$$
Note that all the terms of $a_l\cdot[\tau_\mathbf{d_1}\tau_0^n]_{g^{'}+g_2}^\Theta\cdot[\tau_\mathbf{d_2}]_{g_1-g^{'}}^\Theta$ appears in the expansion of $C_{\mathbf{d^{'}}}$.
\end{proof}

\begin{lemma}\label{Vproduct}
Fix $n_1,n_2,s\geq 0$, then $$\sum_{\substack{g_1+g_2=g\\[3pt]2g_i+n_i\geq s}}V^\Theta_{g_1,n_1}\cdot V_{g_2,n_2}^\Theta=\mathit{O}\left(\frac{V_{g,n_1+n_2}^\Theta}{g^s}\right).$$
\end{lemma}
\begin{proof}
Without loss of generality, suppose $n_1\geq n_2$. By \eqref{N11}, we have 
\begin{align*}\sum_{\substack{g_1+g_2=g\\[3pt]2g_i+n_i\geq s}}V^\Theta_{g_1,n_1}\cdot V_{g_2,n_2}^\Theta&=\mathit{O}\left(\sum_{\substack{g_1+g_2=g, g_1\geq g_2\\[3pt]2g_2+n_2\geq s}}V^\Theta_{g_1,n_1}\cdot V_{g_2,n_2}^\Theta\right)\\&=\mathit{O}\left(\sum_{\substack{g_1+g_2=g, g_1\geq g_2\\[3pt]2g_2+n_2\geq s}}\frac{V^\Theta_{g_1,n_1+s}\cdot V_{g_2,n_2}^\Theta}{g^s}\right).\end{align*}

\noindent(\textbf{Case 1}) If $n_2\geq s$, then  \begin{align*}\sum_{\substack{g_1+g_2=g,g_1\geq g_2\\[3pt]2g_2+n_2\geq s}}V^\Theta_{g_1,n_1+s}\cdot V_{g_2,n_2}^\Theta&=V_{g-1,n_1+s}^\Theta\cdot V^\Theta_{1,n_2}+\sum_{\substack{g_1+g_2=g,g_1\geq g_2\geq 2\\[3pt]2g_2+n_2\geq s}}V^\Theta_{g_1,n_1+s}\cdot V_{g_2,n_2}^\Theta.\end{align*}
Since $n_2$ is fixed, then by \eqref{N11} and \eqref{G11}, we get $$V_{g-1,n_1+s}^\Theta\cdot V_{1,n_2}^\Theta=\mathit{O}\left(V_{g-1,n_1+n_2}^\Theta\right)=\mathit{O}\left(V_{g,n_1+n_2}^\Theta \right).$$
By Lemma \ref{VproductBasic} for $\mathbf{d}=(0,...,0)$ and $2+n_2\geq s+2$, we get the following for $i\geq 2$, $$V_{g-i,n_1+s}^\Theta\cdot V_{i,n_2}^\Theta=V_{g-i,n_1-2+s+2}^\Theta\cdot V_{i,n_2}^\Theta=\mathit{O}\left(V_{g,n_1+n_2-2}^\Theta \right).$$
Therefore, \eqref{N11} implies that
\begin{align*}\sum_{\substack{g_1+g_2=g,g_1\geq g_2\\[3pt]2g_2+n_2\geq s}}V^\Theta_{g_1,n_1+s}\cdot V_{g_2,n_2}^\Theta=\mathit{O}\left(V_{g,n_1+n_2}^\Theta +g\cdot V_{g,n_1+n_2-2}^\Theta \right)=\mathit{O}\left(V_{g,n_1+n_2}^\Theta\right). \end{align*}

\noindent(\textbf{Case 2}) If $n_2<s$, then denote $m:=\lceil\frac{s-n_2}{2} \rceil$. Since \begin{align*}\sum_{\substack{g_1+g_2=g,g_1\geq g_2\\[3pt]2g_2+n_2\geq s}}V^\Theta_{g_1,n_1+s}\cdot V_{g_2,n_2}^\Theta&=V_{g-m,n_1+s}^\Theta\cdot V^\Theta_{m,n_2}\\&\qquad+\sum_{\substack{g_1+g_2=g, g_1\geq g_2\geq m+1\\[3pt]2g_2+n_2\geq s}}V^\Theta_{g_1,n_1+s}\cdot V_{g_2,n_2}^\Theta.\end{align*}
With the same argument of Case 1, we have $$V_{g-m,n_1+s}^\Theta\cdot V_{m,n_2}^\Theta=\mathit{O}\left( V_{g-m,n_1+s}^\Theta\right)=\mathit{O}\left(V_{g,n_1+n_2}^\Theta \right),$$
and when $i\geq m+1$, Lemma \ref{VproductBasic}
 yields
$$V_{g-i,n_1+s}^\Theta\cdot V_{i,n_2}^\Theta=V_{g-i,n_1-2+s+2}^\Theta\cdot V_{i,n_2}^\Theta=\mathit{O}\left(V_{g,n_1+n_2-2}^\Theta \right).$$
Then \eqref{N11} implies that
\begin{align*}\sum_{\substack{g_1+g_2=g,g_1\geq g_2\\[3pt]2g_2+n_2\geq s}}V^\Theta_{g_1,n_1+s}\cdot V_{g_2,n_2}^\Theta=\mathit{O}\left(V_{g,n_1+n_2}^\Theta +g\cdot V_{g,n_1+n_2-2}^\Theta \right)=\mathit{O}\left(V_{g,n_1+n_2}^\Theta\right). \end{align*}
So the proof is completed.
\end{proof}

\begin{remark}
(1) The implied constant in Lemma \ref{Vproduct} only depends on $n_1$ and $n_2$.\\
(2) In particular, when $n_1=n_2=1$, we have
$$\sum\limits_{i=r}\limits^{\lfloor g/2 \rfloor}V_{i,1}^\Theta\cdot V_{g-i,1}^\Theta=\mathit{O}\left(\frac{V_g^\Theta}{g^{2r+1}}\right),\ \ \text{as}\ \ g\to\infty.$$
(3) Let $n_1,n_2\geq 0$, we have
$$\sum_{\substack{g_1+g_2=g\\[3pt]g_1\geq g_2\geq 1}}V^\Theta_{g_1,n_1+1}\times V_{g_2,n_2+1}^\Theta=\mathit{O}\left(\frac{V_{g,n}^\Theta}{g}\right),$$
where $n=n_1+n_2$.
\end{remark}

\vskip 20pt
\section{Algorithm for computing coefficients in the expansions}
\setcounter{equation}{0} 
In this section, we give a detailed algorithm which can calculate the coefficients in \eqref{Tauasymp}-\eqref{Gasymp} explicitly up to any given order and then we use it to compute first three terms.

\noindent\textbf{Notation.} If $f$ is a polynomial of in variables $x_1,...,x_n$ with coefficients in a field $\mathbb{F}$, then we denote the coefficient of the monomial $x_1^{\alpha_1}\cdots x_n^{\alpha_n}$ in $f(x_1,...,x_n)$ by
$$[x_1^{\alpha_1}\cdots x_n^{\alpha_n}]f.$$

\subsection{Calculation algorithm.}
In view of Fact 3, if we know the expansions of the ratio $ \frac{4V_{g-1,n+2}^\Theta}{V_{g,n}^\Theta}$ in the inverse powers of $g$ up to $\mathit{O}(1/g^s)$, then the expansions of $\frac{V_{g,n}^\Theta}{4V_{g-1,n+2}^\Theta}$ in the inverse powers of $g$ up to $\mathit{O}(1/g^s)$ is obtained. The reverse is also true. Combining with this, we give an algorithm to calculate the coefficients of \eqref{Tauasymp}-\eqref{Gasymp}. In brief, first, the expansions of the ratios $\frac{4V_{g-1,n+2}^\Theta}{V_{g,n}^\Theta}$, $\frac{\frac{\pi}{2}(2g-2+n)V_{g,n}^\Theta}{V_{g,n+1}^\Theta}$ up to $\mathit{O}(1/g^s)$ can be obtained by the expansions of  $ \frac{4^{|\mathbf{d}|}[\tau_{d_1}\cdots\tau_{d_n}]^\Theta_{g}}{V_{g,n}^\Theta}$ up to $\mathit{O}(1/g^{s})$. Second, the expansions of  $\frac{4^{|\mathbf{d}|}[\tau_{d_1}\cdots\tau_{d_n}]^\Theta_{g}}{V_{g,n}^\Theta}$ up to $\mathit{O}(1/g^{s})$ can be obtained by the expansions $\frac{4^{|\mathbf{d}|}[\tau_{d_1}\cdots\tau_{d_n}]^\Theta_{g}}{V_{g,n}^\Theta}$, $\frac{4V_{g-1,n+2}^\Theta}{V_{g,n}^\Theta}$ and $\frac{\frac{\pi}{2}(2g-2+n)V_{g,n}^\Theta}{V_{g,n+1}^\Theta}$ up to $\mathit{O}(1/g^{s-1})$. We can obtain the coefficients up to any given order by repeating this two steps.

\noindent\textbf{Algorithm.} For any fixed $s,n$, the expansion of $\frac{4V_{g-1,n+2}^\Theta}{V_{g,n}^\Theta}$, $\frac{\frac{\pi}{2}(2g-2+n)V_{g,n}^\Theta}{V_{g,n+1}^\Theta}$ and $ \frac{4^{|\mathbf{d}|}[\tau_{d_1}\cdots\tau_{d_n}]^\Theta_{g}}{V_{g,n}^\Theta}$ in the inverse powers of $g$ up to $\mathit{O}(1/g^s)$ can be obtained in the following four steps:
\begin{itemize}[leftmargin=2em]
\item [1.] In general, given $g^{'},n^{'}$, in order to obtain the expansion of $ \frac{V_{g-g^{'},n-n^{'}}^\Theta}{V_{g,n}^\Theta}$ up to $\mathit{O}(1/g^s)$, it is enough to know the expansions of $\frac{4V_{g-1,k+2}^\Theta}{V_{g,k}^\Theta}$ and $ \frac{\frac{\pi}{2}(2g-2+k)V_{g,k}^\Theta}{V_{g,k+1}^\Theta}$ up to $\mathit{O}(1/g^{s-2g^{'}-n^{'}})$, since
\begin{equation}\label{Vratio2}\begin{aligned}
\frac{V_{g-g^{'},n-n^{'}}^\Theta}{V_{g,n}^\Theta}&=\frac{1}{4^{g^{'}}}\times\prod\limits_{j=-n^{'}+1}\limits^{2g^{'}}\frac{\frac{\pi}{2}(2g-2g^{'}+n+j-3)V_{g-g^{'},n+j-1}^\Theta}{V_{g-g^{'},n+j}^\Theta}\\&\qquad\times\prod\limits_{j=1}\limits^{g^{'}}\frac{4V_{g-j,n+2j}^\Theta}{V_{g-j+1,n+2j-2}^\Theta}\times\prod\limits_{j=-n^{'}+1}\limits^{2g^{'}}\frac{1}{\frac{\pi}{2}(2g-2g^{'}+n+j-3)}.
\end{aligned}
\end{equation}

\item [2.] Following recursion $(\mathbf{\Rmnum{2}})$, we have 
\begin{equation}\label{Nre}
\frac{\frac{\pi}{2}(2g-2+n)V_{g,n}^\Theta}{V_{g,n+1}^\Theta}=\sum\limits_{L=0}\limits^{g-1}\frac{(-1)^L(\pi/2)^{2L+1}}{(2L+1)!}\times\frac{4^L[\tau_{L}\tau_0^{n}]^\Theta_g}{V_{g,n+1}^\Theta}.
\end{equation}
So in view of Fact 1, the expansion of $\frac{\frac{\pi}{2}(2g-2+n)V_{g,n}^\Theta}{V_{g,n+1}^\Theta}$ up to $\mathit{O}(1/g^s)$ can be written explicitly in terms of the expansion of $ \frac{4^{k}[\tau_{k}\tau_0^n]^\Theta_{g}}{V_{g,n+1}^\Theta}$ up to $\mathit{O}(1/g^s)$.

\item [3.]Similarly, recursion $(\mathbf{\Rmnum{1}})$ implies
\begin{equation}\label{Gre}
\frac{4V_{g-1,n+4}^\Theta}{V_{g,n+2}^\Theta}=\frac{4[\tau_1\tau_0^{n+1}]_g^\Theta}{V_{g,n+2}^\Theta}-\frac{24}{V_{g,n+2}^\Theta}\sum_{\substack{g_1+g_2=g\\[3pt]I\sqcup J=\{1,...,n\}}}V^{\Theta}_{g_1,|I|+2}\times V^{\Theta}_{g_2,|J|+2}.
\end{equation}
Then by Lemma \ref{Vproduct}, the expansion of $\frac{4V_{g-1,n+4}^\Theta}{V_{g,n+2}^\Theta}$ up to $\mathit{O}(1/g^s)$ can be written explicitly in terms of the expansion of $ \frac{4[\tau_1\tau_0^{n+1}]^\Theta_{g}}{V_{g,n+1}^\Theta}$ up to $\mathit{O}(1/g^s)$ and the expansion of $$\frac{1}{V^\Theta_{g,n+2}}\sum_{\substack{g_1+g_2=g\\[3pt]I\sqcup J=\{1,...,n\}}}V^\Theta_{g_1,|I|+2}\times V_{g_2,|J|+2}^\Theta$$ up to $\mathit{O}(1/g^s)$ with either $ 2g_1+|I|<s$ or $ 2g_2+|J|<s$. Then we use \eqref{Vratio2} to get the expansions of $\frac{4V_{g-1,3}^\Theta}{V_{g,1}^\Theta}$ and $\frac{4V_{g-1,2}^\Theta}{V_{g,0}^\Theta}$.

\item [4.] The expansion of $\frac{4^{|\mathbf{d}|}[\tau_{d_1}\cdots\tau_{d_n}]^\Theta_{g}}{V_{g,n}^\Theta}$ up to $\mathit{O}(1/g^s)$ is explicitly obtained by the expansion of $\frac{4^{|\mathbf{k}|}\left([\tau_{k_1}\tau_{k_2}\cdots\tau_{k_n}]^\Theta_{g}-4[\tau_{k_1+1}\tau_{k_2}\cdots\tau_{k_n}]_{g}^\Theta\right)}{V_{g,n}^\Theta}$ up to $\mathit{O}(1/g^s)$, since \begin{equation}\label{summation}1-\frac{4^{|\mathbf{d}|}[\tau_{d_1}\cdots\tau_{d_n}]^\Theta_{g}}{V_{g,n}^\Theta}=\frac{\sum\limits_{|\mathbf{k}|=0}\limits^{|\mathbf{d}|-1}4^{|\mathbf{k}|}\left([\tau_{k_1}\tau_{k_2}\cdots\tau_{k_n}]^\Theta_{g}-4[\tau_{k_1+1}\tau_{k_2}\cdots\tau_{k_n}]\right)}{V_{g,n}^\Theta},\end{equation} 
(the right hand side of \eqref{summation} is an increasing summation from $(0,...,0)$ to $(d_1,...,d_n)$ in a lexicographic order). Recursion $(\mathbf{\Rmnum{3}})$ implies
\begin{equation}\label{difference}\frac{4^{|\mathbf{k}|}\left([\tau_{k_1}\tau_{k_2}\cdots\tau_{k_n}]^\Theta_{g}-4[\tau_{k_1+1}\tau_{k_2}\cdots\tau_{k_n}]_{g}^\Theta\right)}{V_{g,n}^\Theta}=V_1+V_2+V_3,\end{equation}
where \begin{equation}\label{V1}V_1=\frac{1}{\frac{\pi}{2}(2g-3+n)}\cdot\frac{\frac{\pi}{2}(2g-3+n)V_{g,n-1}^\Theta}{V_{g,n}^\Theta}\cdot\frac{4^{|\mathbf{k}|}A_{\mathbf{k},g,n}}{V_{g,n-1}^\Theta},\end{equation}
\begin{equation}\label{V2}V_2=\frac{1}{\frac{\pi}{2}(2g-3+n)}\cdot\frac{\frac{\pi}{2}(2g-3+n)V_{g-1,n+1}^\Theta}{V_{g-1,n+2}^\Theta}\cdot\frac{4V_{g-1,n+2}^\Theta}{V_{g,n}^\Theta}\cdot\frac{4^{|\mathbf{k}|-1}B_{\mathbf{k},g,n}}{V_{g-1,n+1}^\Theta},\end{equation}
\begin{equation}\label{V3}V_3=\sum_{\substack{1\leq g^{'}\leq g-1\\[3pt]I\sqcup J=\{2,...,n\}}}\frac{ V_{g^{'},|I|+1}^\Theta\cdot V_{g-g^{'},|J|+1}^\Theta}{V_{g,n}^\Theta}\cdot\frac{4^{|\mathbf{k}|}C_{\mathbf{k},g,n}}{V_{g^{'},|I|+1}^\Theta\cdot V_{g-g^{'},|J|+1}^\Theta},\end{equation}
and
\begin{equation}\label{AV}
\frac{4^{|\mathbf{k}|}A_{\mathbf{k},g,n}}{V_{g,n-1}^\Theta}=\sum\limits_{j=2}\limits^{n}(2k_j+1)\sum\limits_{L=0}\limits^{g-1-|\mathbf{k}|}\left(\frac{a_{L}}{4^{L}}-\frac{a_{L-1}}{4^{L-1}}\right)\frac{4^{|\mathbf{k}|+L}[\tau_{k_1+k_j+L}\tau_{k_2}\cdots\widehat{\tau_{k_j}}\cdots\tau_{k_n}]_{g}^\Theta}{V_{g,n-1}^\Theta}
\end{equation}
(here the hat means that the corresponding entry is empty and $a_{-1}=0$),
\begin{equation}\label{BV}
\frac{4^{|\mathbf{k}|-1}B_{\mathbf{k},g,n}}{V_{g-1,n+1}^\Theta}=2\left(\sum\limits_{L=0}\limits^{g-1-|\mathbf{k}|}\Big(\frac{a_{L}}{4^{L}}-\frac{a_{L-1}}{4^{L-1}}\Big)\frac{\sum_{c_1+c_2=k_1+L-1}4^{|\mathbf{k}|+L-1}[\tau_{c_1}\tau_{c_2}\tau_{k_3}\cdots\tau_{k_n}]_{g}^\Theta}{V_{g-1,n+1}^\Theta}\right),
\end{equation}

\begin{equation}\label{CV}\begin{aligned}
&\frac{4^{|\mathbf{k}|}C_{\mathbf{k},g,n}}{V_{g_1,|I|+1}^\Theta\cdot V_{g_2,|J|+1}^\Theta}=2\sum\limits_{L=0}\limits^{g-1-|\mathbf{k}|}\Big(\frac{a_{L}}{4^{L}}-\frac{a_{L-1}}{4^{L-1}}\Big)\\&\times\left(\frac{\sum\limits_{c_1+c_2=L+k_1-1}4^{L+|\mathbf{k}|-1}[\tau_{c_1}\prod\limits_{i\in I}\tau_{k_i}]_{g^{'}}^\Theta\cdot[\tau_{c_2}\prod\limits_{j\in J}\tau_{k_j}]_{g-g^{'}}^\Theta}{V_{g^{'},|I|+1}^\Theta\cdot V_{g-g^{'},|J|+1}^\Theta}\right).\end{aligned}
\end{equation}
From \eqref{summation} and \eqref{difference}, the expansion of $\frac{4^{|\mathbf{d}|}[\tau_{d_1}\cdots\tau_{d_n}]^\Theta_{g}}{V_{g,n}^\Theta}$ up to $\mathit{O}(1/g^s)$ can be derived from three parts $V_1-V_3$ up to $\mathit{O}(1/g^s)$. Now we analyze the contributions from \eqref{V1}-\eqref{V3} one by one.

\noindent$\bullet$\textbf{Contribution from $V_1$.} By \eqref{V1}, the expansion of $V_1$ up to $\mathit{O}(1/g^s)$ can be obtained by the expansion of $\frac{\frac{\pi}{2}(2g-3+n)V_{g,n-1}^\Theta}{V_{g,n}^\Theta}$ and $\frac{4^{|\mathbf{k}|}A_{\mathbf{k},g,n}}{V_{g,n-1}^\Theta}$ up to $\mathit{O}(1/g^{s-1})$. The latter can be obtained by the expansions for $\frac{4^{|\mathbf{d}|}[\tau_{d_1}\cdots\tau_{d_n}]^\Theta_{g}}{V_{g,n}^\Theta}$ up to $\mathit{O}(1/g^{s-1})$ via \eqref{AV} in view of Fact 1.

\noindent$\bullet$\textbf{Contribution from $V_2$.} Similarly, the expansion of $V_2$ up to $\mathit{O}(1/g^s)$ can be obtained by the expansion of $\frac{\frac{\pi}{2}(2g-3+n)V_{g-1,n+1}^\Theta}{V_{g-1,n+2}^\Theta}$, $\frac{4V_{g-1,n+2}^\Theta}{V_{g,n}^\Theta}$ and $\frac{4^{|\mathbf{k}|-1}B_{\mathbf{k},g,n}}{V_{g-1,n+1}^\Theta}$ up to $\mathit{O}(1/g^{s-1})$. The last one can be obtained by the expansions for $\frac{4^{|\mathbf{d}|}[\tau_{d_1}\cdots\tau_{d_n}]^\Theta_{g}}{V_{g,n}^\Theta}$ up to $\mathit{O}(1/g^{s-1})$ via \eqref{BV} in view of Fact 1.

\noindent$\bullet$\textbf{Contribution from $V_3$.} First, Lemma \ref{Vproduct} implies that 
$$\frac{\sum_{\substack{1\leq g^{'}\leq g-1\\[3pt]I\sqcup J=\{2,...,n\}}} V_{g^{'},|I|+1}^\Theta\cdot V_{g-g^{'},|J|+1}^\Theta}{V_{g,n}^\Theta}$$
is of order at least $\mathit{O}(1/g^{2})$ and it is of order $\mathit{O}(1/g^{s+1})$ unless either $2g^{'}+|I|\leq s$ or $2(g-g^{'})+|J|\leq s$. So in order to get the expansions of $V_3$ up to $\mathit{O}(1/g^{s})$, we need to apply the expansions of $\frac{4V_{g-1,k+2}^\Theta}{V_{g,k}^\Theta}$ and $ \frac{\frac{\pi}{2}(2g-2+k)V_{g,k}^\Theta}{V_{g,k+1}^\Theta}$ up to $\mathit{O}(1/g^{s-2g^{'}-n^{'}})$ in \eqref{Vratio2} and we also need to apply the expansions for $\frac{4^{|\mathbf{d}|}[\tau_{d_1}\cdots\tau_{d_n}]^\Theta_{g}}{V_{g,n}^\Theta}$ up to $\mathit{O}(1/g^{s-2})$ in \eqref{CV} with the help of Fact 1 (cf. Appendix A).
\end{itemize}

\subsection{First three terms} Now we use the algorithm above to compute first three terms in \eqref{Tauasymp}-\eqref{Gasymp}. 
\begin{theorem} As $g\to\infty$,\\
(1). for any fixed $n\geq 1$ and $\mathbf{d}=(d_1,...,d_n)$, we have
\begin{equation}\label{tauasymp2}
\frac{4^{k}[\prod\limits_{i=1}\limits^{n}\tau_{d_i}]_g^\Theta}{V_{g,n}^\Theta}=1+\frac{e^1_{n,\mathbf{d}}}{g}+\frac{e^2_{n,\mathbf{d}}}{g^2}+\mathit{O}\left(\frac{1}{g^{3}}\right),
\end{equation}where $$e^1_{n,\mathbf{d}}=-\frac{4|\mathbf{d}|(|\mathbf{d}|+n-3/2)}{\pi^2},$$ and \begin{align*}e^2_{n,\mathbf{d}}&=\frac{1}{\pi^4}\times\Bigg[8|\mathbf{d}|^4+(16n-40)|\mathbf{d}|^3+\Big(8n^2+(2\pi^2-48)n-6\pi^2+62\Big)|\mathbf{d}|^2\\&\qquad+\Big((2\pi^2-12)n^2-(9\pi^2-44)n-38+9\pi^2-\frac{\pi}{4}\Big)|\mathbf{d}|\Bigg]+\frac{n-s}{16\pi^2},\end{align*}
where $s:=\#\{i|d_i=0\}$ denotes the number of zero in $\mathbf{d}$.

(2). for any fixed $n\geq 0$, the following holds
\begin{equation}\label{Nasymp2}
\frac{\frac{\pi}{2}(2g-2+n)V^\Theta_{g,n}}{V^\Theta_{g,n+1}}=1+\frac{a^1_{n}}{g}+\frac{a^2_{n}}{g^2}+\mathit{O}\left(\frac{1}{g^{3}}\right),
\end{equation}
and
\begin{equation}\label{Gasymp2}
\frac{4V^\Theta_{g-1,n+2}}{V^\Theta_{g,n}}=1+\frac{b^1_{n}}{g}+\frac{b^2_{n}}{g^2}+\mathit{O}\left(\frac{1}{g^2}\right),
\end{equation}
where $$ a^1_{n}=\frac{8n-8+\pi^2}{4\pi^2},\qquad b^1_{n}=-\frac{4n-2}{\pi^2},$$ and $$a^2_{n}=\frac{1}{\pi^4}\left[\Big(-\frac{3\pi^2}{2}+8 \Big)n^2-\Big(\frac{\pi^4}{8}-5\pi^2+20\Big)n+\frac{17\pi^4}{64}-\frac{\pi^3}{16}-\frac{27\pi^2}{8}+\frac{\pi}{4}+12 \right],$$ $$b^2_n=\frac{1}{\pi^4}\left[(2\pi^2-4)n^2-(7\pi^2-12)n+\frac{13\pi^2}{8}-\frac{\pi}{2}-8\right].$$
\end{theorem}
The proof of Theorem 3.1 is divided into three parts and relies on the following two remarks.

\begin{remark}
$\displaystyle \sum\limits_{L=0}\limits^{\infty}(-1)^L\frac{(\pi/2)^{2L+1}L^k}{(2L+1)!}$ is a polynomial in $\pi^2$ of degree at most $\lfloor k/2 \rfloor$ with rational coefficients, since
$$\sum\limits_{L=0}\limits^{\infty}(-1)^L\frac{(\pi/2)^{2L+1}L^k}{(2L+1)!}=\left[xD^k\Big(\frac{\mathrm{sin}x}{x}\Big)\right]_{x=\frac{\pi}{2}},$$
where $\displaystyle D=\frac{x}{2}\cdot\frac{d}{dx}.$ For example, when $k=1$, we have $$\sum\limits_{L=0}\limits^{\infty}(-1)^L\frac{(\pi/2)^{2L+1}L}{(2L+1)!}=-\frac{1}{2},$$
and when $k=2$, $$\sum\limits_{L=0}\limits^{\infty}(-1)^L\frac{(\pi/2)^{2L+1}L^2}{(2L+1)!}=-\frac{\pi^2}{16}+\frac{1}{4}.$$
\end{remark}

\begin{remark}
From the proof of part (4) in the Lemma \ref{Basic-est}, we have \begin{equation}\label{j2}\begin{aligned}\sum\limits_{i=0}\limits^\infty(i+1)^2 \bigg(\frac{a_{i+1}}{4^{i+1}}-\frac{a_i}{4^i}\bigg)&=\frac{1}{\pi}\int_{0}^\infty \frac{e^{-x/2}}{(\mathrm{cosh}\frac{x}{2})^2}\Big(D^2(\mathrm{cosh}\frac{x}{2}-1)+F^2(\mathrm{sinh}\frac{x}{2})\Big)\\&=\frac{1}{\pi}\times\Big(\frac{1}{2}+\frac{\pi^2}{12}\Big),\end{aligned}\end{equation} and
\begin{equation}\label{j3}\begin{aligned}\sum\limits_{i=0}\limits^\infty(i+1)^3 \bigg(\frac{a_{i+1}}{4^{i+1}}-\frac{a_i}{4^i}\bigg)&=\frac{1}{\pi}\int_{0}^\infty \frac{e^{-x/2}}{(\mathrm{cosh}\frac{x}{2})^2}\Big(D^3(\mathrm{cosh}\frac{x}{2}-1)+F^3(\mathrm{sinh}\frac{x}{2})\Big)\\&=\frac{1}{\pi}\times\Big(\frac{1}{4}+\frac{3\pi^2}{16}\Big).\end{aligned}\end{equation}
\end{remark}

\subsubsection{Constant term} Lemma \ref{e1-est} immediately implies that $e^0_{n,\mathbf{d}}$ equals 1. Then we prove the constant terms in $a_n^0$ and $b_n^0$ in \eqref{Nasymp2} and \eqref{Gasymp2} respectively via the second and the third step of Algorithm (cf. Section 3.1).

\noindent\textbf{Proof.}
In view of Fact 1 (cf. Appendix A), let $$r_i=(-1)^i\frac{(\pi/2)^{2i+1}}{(2i+1)!}, k_g=g-1, c_i=1\ \text{and}\ c_{g,i}=\frac{4^{i+1}[\tau_{i+1}\tau_0^{n}]_g^\Theta}{V_{g,n+1}^\Theta}.$$
Then by \eqref{Fact1} and Lemma \ref{e1-est}, we get the following from \eqref{Nre} 
$$a_n^0=\lim\limits_{g\to\infty}\frac{\frac{\pi}{2}(2g-2+n)V_{g,n}^\Theta}{V_{g,n+1}^\Theta}=\sum\limits_{i=0}^{\infty}(-1)^i\frac{(\pi/2)^{2i+1}}{(2i+1)!}=1.$$ 

On the other hand, as $g\to\infty$, by Lemma \ref{e1-est} together with \eqref{N11} and \eqref{Times1}, we obtain the following via \eqref{Gre} for $n\geq 0$,
$$\frac{4V_{g-1,n+4}^\Theta}{V_{g,n+2}^\Theta}=\frac{4[\tau_{1}\tau_0^{n+1}]_{g-1}^\Theta}{V_{g,n+2}^\Theta}+\mathit{O}\left(\frac{1}{g^2}\right)=1+\mathit{O}\left(\frac{1}{g^2}\right).$$
So $b_n^0$ in \eqref{Gasymp2} coincides with 1 for $n\geq 2$. Now by $a_n^0$ in \eqref{Nasymp} equals $1$ for $n\geq 0$, it is easy to check the $b_n^0$ in \eqref{Gasymp2} for $n=0,1$ by \eqref{Vratio2},
\begin{equation}\label{n1}\frac{4V_{g-1,3}^\Theta}{V_{g,1}^\Theta}=\frac{\frac{\pi}{2}(2g-1)V_{g-1,3}^\Theta}{V_{g-1,4}^\Theta}\cdot\frac{4V_{g-1,4}^\Theta}{V_{g,2}^\Theta}\cdot\frac{V_{g,2}^\Theta}{\frac{\pi}{2}(2g-1)V_{g,1}^\Theta},\end{equation}
and \begin{equation}\label{n0}\frac{4V_{g-1,2}^\Theta}{V_{g,0}^\Theta}=\frac{\frac{\pi}{2}(2g-2)V_{g-1,2}^\Theta}{V_{g-1,3}^\Theta}\cdot\frac{4V_{g-1,3}^\Theta}{V_{g,1}^\Theta}\cdot\frac{V_{g,1}^\Theta}{\frac{\pi}{2}(2g-1)V_{g,0}^\Theta}.\end{equation}
So $a_n^0$ and $b_n^0$ in \eqref{Nasymp2}-\eqref{Gasymp2} both coincide with 1. \qed

\subsubsection{Coefficients of $\frac{1}{g}$.}Now from constant term $e^0_{n,\mathbf{d}}$, $a^0_{n}$ and $b^0_{n}$ in \eqref{tauasymp2}-\eqref{Gasymp2}, we calculate $e^1_{n,\mathbf{d}}$, $a^1_{n}$ and $b^1_{n}$ in \eqref{tauasymp2}-\eqref{Gasymp2}, respectively via the Algorithm in Section 3.1. First, we remark the numerical properties which will be used in the following proof.

\noindent\textbf{Proof.}
First, we prove \begin{equation}\label{e1}e^1_{n,\mathbf{d}}=-\frac{4|\mathbf{d}|(|\mathbf{d}|+n-3/2)}{\pi^2}.\end{equation} In view of the fourth step in Algorithm (cf. Section 3.1), it is enough to evaluate contributions from $V_1$, $V_2$ and $V_3$. There are two cases in \eqref{difference}. Note that the contribution from $V_3$ in the fourth step of Algorithm shows 
$$\Big[\frac{1}{g}\Big]V_3=0.$$

\noindent\textbf{(Case 1)} For any $\mathbf{d}=(d_1,...,d_s,0,...,0)$ with $1\leq s\leq n-1$, we compute $$\frac{4^{|\mathbf{d}|}\left([\tau_{d_1}\cdots\tau_{d_s}\tau_0^{n-s}]^\Theta_{g}-4[\tau_{d_1}\cdots\tau_{d_s}\tau_1\tau_0^{n-s-1}]^\Theta_{g}\right)}{V_{g,n}^\Theta}$$

From \eqref{V1}-\eqref{V2}, then by Lemma \ref{e1-est}, Fact 1 (cf. Appendix A) and Lemma \ref{Coeff-est} (2)-(3), we get the following as $g\to\infty$,
\begin{align*}\Big[\frac{1}{g}\Big]V_1&=\frac{1}{\pi}\left[\sum\limits_{j=2}\limits^{n}(2d_j+1)\sum\limits_{L=0}\limits^{\infty}\Big(\frac{a_{L}}{4^{L}}-\frac{a_{L-1}}{4^{L-1}}\Big)+(n-1-s)\sum\limits_{L=0}\limits^{\infty}\Big(\frac{a_{L}}{4^{L}}-\frac{a_{L-1}}{4^{L-1}}\Big)\right]\\&=\frac{8|\mathbf{d}|+4n-4}{\pi^2}\end{align*}
and
\begin{align*}\Big[\frac{1}{g}\Big]V_2=\frac{1}{\pi}\left[2\sum\limits_{L=0}\limits^{\infty}L\Big(\frac{a_{L}}{4^{L}}-\frac{a_{L-1}}{4^{L-1}}\Big)\right]=\frac{2}{\pi}.\end{align*}

\noindent\textbf{(Case 2)} For any $\mathbf{d}=(d_1,...,d_s,k,...,0)$ with $1\leq s\leq n-1$ and $k\geq 1$, we compute $$\left[\frac{1}{g}\right]\left(\frac{4^{|\mathbf{d}|}\left([\tau_{d_1}\cdots\tau_{d_s}\tau_{k}\tau_0^{n-s-1}]^\Theta_{g}-4[\tau_{d_1}\cdots\tau_{d_s}\tau_{k+1}\tau_0^{n-s-1}]^\Theta_{g}\right)}{V_{g,n}^\Theta}\right)$$
With the same argument in Case 1, we get the following as $g\to\infty$,
\begin{align*}\Big[\frac{1}{g}\Big]V_1&=\frac{1}{\pi}\left[\sum\limits_{j=2}\limits^{n}(2d_j+1)\sum\limits_{L=0}\limits^{\infty}\Big(\frac{a_{L}}{4^{L}}-\frac{a_{L-1}}{4^{L-1}}\Big)+(n-1-s)\sum\limits_{L=0}\limits^{\infty}\Big(\frac{a_{L}}{4^{L}}-\frac{a_{L-1}}{4^{L-1}}\Big)\right]\\&=\frac{8|\mathbf{d}|-8k+4n-4}{\pi^2}\end{align*}
and
\begin{align*}\Big[\frac{1}{g}\Big]V_2=\frac{1}{\pi}\left[2\sum\limits_{L=0}\limits^{\infty}(L+k)\Big(\frac{a_{L}}{4^{L}}-\frac{a_{L-1}}{4^{L-1}}\Big)\right]=\frac{8k+2}{\pi}.\end{align*}

Hence, as $g\to\infty$, we get
$$\Big[\frac{1}{g}\Big]V_1=\frac{4\left(n-1+2|\mathbf{d}|-2d_1\right)}{\pi^2},\qquad\Big[\frac{1}{g}\Big]V_2=\frac{8d_1+2}{\pi^2},\qquad \Big[\frac{1}{g}\Big]V_3=0.$$
Then, as $g\to\infty$,$$\left[\frac{1}{g}\right]\left(\frac{4^{|\mathbf{d}|}\left([\tau_{d_1}\prod\limits_{i=2}\limits^{n}d_i]_g^\Theta-4[\tau_{d_1+1}\prod\limits_{i=2}\limits^{n}d_i]_g^\Theta\right)}{V_{g,n}^\Theta}\right)=\frac{4n-2+8|\mathbf{d}|}{\pi^2}.$$
So \eqref{e1} follows by \eqref{summation}.

Now we calculate $a_n^1$ and $b_n^1$ in \eqref{Nasymp2} and \eqref{Gasymp2} respectively via the second and the third step of Algorithm (cf. Section 3.1). By part \eqref{e1}, Fact 1 (cf. Appendix A) and Remark 3.2, \eqref{Nre} implies that as $g\to\infty$,
\begin{align*}
\left[\frac{1}{g}\right]\left(\frac{\frac{\pi}{2}(2g-2+n)V^\Theta_{g,n}}{V^\Theta_{g,n+1}}\right)&=\sum\limits_{L=0}\limits^{\infty}\frac{(-1)^L(\pi/2)^{2L+1}}{(2L+1)!}\times\left[\frac{1}{g}\right]\left(\frac{4^L[\tau_{L}\tau_0^{n}]^\Theta_g}{V_{g,n+1}^\Theta} \right)\\&=-\sum\limits_{L=0}\limits^{\infty}\frac{(-1)^L(\pi/2)^{2L+1}}{(2L+1)!}\cdot\frac{4L^2+4(n-\frac{1}{2})L}{\pi^2}\\&=\frac{8n-8+\pi^2}{4\pi^2}.
\end{align*}
Similarly, we obtain $b_n^1$ for $n\geq 2$. Note that \eqref{e1} shows as $g\to\infty$,
$$\frac{4[\tau_{1}\tau_0^{n+1}]^\Theta_g}{V_{g,n+2}^\Theta}=1-\frac{4n+6}{\pi^2}\times\frac{1}{g}+\mathit{O}\left(\frac{1}{g^2}\right).$$
By \eqref{N11} and Lemma \ref{Vproduct} for $g_1,g_2\geq 1$ and fixed $n$, $$\sum_{\substack{g_1+g_2=g\\[3pt]I\sqcup J=\{1,...,n\}}}V_{g_1,|I|+2}^\Theta\times V_{g_2,|J|+2}^\Theta=\mathit{O}\left(\frac{V_{g,n+4}^\Theta}{g^4}\right)=\mathit{O}\left(\frac{V_{g,n+2}^\Theta}{g^2}\right).$$
So by \eqref{Gre}, we obtain $$\frac{4V^\Theta_{g-1,n+4}}{V^\Theta_{g,n+2}}=1-\frac{4n+6}{\pi^2}\times\frac{1}{g}+\mathit{O}\left(\frac{1}{g^2}\right),$$ as $g\to\infty$. 
The remaining cases $b_n^1$ for $n=0$ and $n=1$ can be verified from $a_n^1$ for any $n$ and $b_n^1$ for $n\geq 2$ via \eqref{n1} and \eqref{n0}. \qed

\subsubsection{Coefficients of $\frac{1}{g^2}$.} From $e^i_{n,\mathbf{d}}$, $a^i_{n}$ and $b^i_{n}$ for $i=0,1$ in \eqref{tauasymp2}-\eqref{Gasymp2}, we can get the terms of order 2 via the Algorithm in Section 3.1 to complete the proof of Theorem 3.1.

\noindent\textbf{Proof.} First, we get $e^2_{n,\mathbf{d}}$ by analyzing the contributions from $V_1$-$V_3$ in \eqref{V1}-\eqref{V3} respectively. Then by $e^2_{n,\mathbf{d}}$ we compute $a_n^2$ and $b_n^2$ in \eqref{Nasymp2} and \eqref{Gasymp2}, respectively. For $\mathbf{d}=(d_1,...,d_s,k,0,...,0)$ with $0\leq s\leq n-1$ and $k\geq 1$, we compute the following as $g\to\infty$,
$$\left[\frac{1}{g^2}\right]\left(\frac{4^{|\mathbf{d}|}\left([\tau_{d_1}\cdots\tau_{d_s}\tau_{k}\tau_0^{n-s-1}]^\Theta_{g}-4[\tau_{d_1}\cdots\tau_{d_s}\tau_{k+1}\tau_0^{n-s-1}]^\Theta_{g}\right)}{V_{g,n}^\Theta}\right).$$

\noindent$\bullet$ \textbf{Contribution from} $V_1$. 
By Remark 3.2 and \eqref{e1}, we get the following as $g\to\infty$,
\begin{align*}
&\Big[\frac{1}{g}\Big]\left(\frac{4^{|\mathbf{d}|}A_{\mathbf{d},g,n}}{V_{g,n-1}^\Theta}\right)\\&=\sum\limits_{i=1}\limits^{n}(2d_i+1)\sum\limits_{L=0}\limits^{\infty}\Big(\frac{a_{L}}{4^{L}}-\frac{a_{L-1}}{4^{L-1}}\Big)\left(-\frac{4(|\mathbf{d}|+L)(|\mathbf{d}|+L+n-5/2)}{\pi^2}\right)\\&\qquad+(n-1-s)\sum\limits_{L=0}\limits^{\infty}\Big(\frac{a_{L}}{4^{L}}-\frac{a_{L-1}}{4^{L-1}}\Big)\left(-\frac{4(|\mathbf{d}|+L)(|\mathbf{d}|+L+n-5/2)}{\pi^2}\right)\\&=-\frac{8|\mathbf{d}|-8k+4n-4}{\pi^3}\left[4|\mathbf{d}|^2+(4n-8)|\mathbf{d}|+n-2+\frac{\pi^2}{12}\right].
\end{align*}
Therefore, as $g\to\infty$, by $a_n^1$ and \eqref{V1}, we have
\begin{equation}\label{V12}
\Big[\frac{1}{g^2}\Big]V_1=-\frac{8|\mathbf{d}|-8k+4n-4}{\pi^4}\times\left[4|\mathbf{d}|^2+(4n-8)|\mathbf{d}|-n+2+\Big(\frac{n}{2}-\frac{5}{3}\Big)\pi^2\right].
\end{equation}

\noindent$\bullet$ \textbf{Contribution from} $V_2$. Withe the same argument above, we get 
\begin{align*}
\Big[\frac{1}{g}\Big]\left(\frac{4^{|\mathbf{d}|}B_{\mathbf{d},g,n}}{V_{g,n-1}^\Theta}\right)&=2\sum\limits_{L=0}\limits^{\infty}\Big(\frac{a_{L}}{4^{L}}-\frac{a_{L-1}}{4^{L-1}}\Big)(L+k)\left(-\frac{4(|\mathbf{d}|+L)(|\mathbf{d}|+L+n-3/2)}{\pi^2}\right)\\&=-\frac{8}{\pi^3}\left[|\mathbf{d}|^2+\Big(n+\frac{\pi^2}{6}-\frac{3}{2}\Big)|\mathbf{d}|+\Big(\frac{\pi^2}{12}-\frac{1}{2}\Big)n+\frac{1}{2}-\frac{\pi^2}{48}\right]\\&\qquad-\frac{8k}{\pi^3}\left[4|\mathbf{d}|^2+(4n-8)|\mathbf{d}|+\frac{\pi^2}{12}+4-3n\right],
\end{align*}
and then by \eqref{V2}, $a_n^1$ and $b_n^1$, we have 
\begin{equation}\label{V22}
\begin{aligned}
\Big[\frac{1}{g^2}\Big]V_2&=-\frac{8k}{\pi^4}\times\left[4|\mathbf{d}|^2+(4n-8)|\mathbf{d}|-n+2+\Big(\frac{n}{2}-\frac{5}{3}\Big)\pi^2\right]\\&\qquad-\frac{8}{\pi^4}\left[|\mathbf{d}|^2+\Big(n+\frac{\pi^2}{6}-\frac{3}{2}\Big)|\mathbf{d}|+\frac{5\pi^2}{24}n-\frac{11\pi^2}{24}\right].
\end{aligned}
\end{equation}

\noindent$\bullet$ \textbf{Contribution from} $V_3$. For fixed n, with the help of \eqref{N1}, Lemma \ref{Vproduct} implies that 
\begin{align*}
\sum_{\substack{g_1+g_2=g\\[3pt]I\sqcup J=\{2,...,n\}}}V^\Theta_{g_1,|I|+1}\times V_{g_2,|J|+1}^\Theta&=2V_{g-1,n}^\Theta\times V_{1,1}^\Theta+4\sum_{\substack{g_1+g_2=g\\[3pt]I\sqcup J=\{2,...,n\}\\[3pt] |I|\geq |J|\geq 1}}V^\Theta_{g_1,|I|+1}\times V_{g_2,|J|+1}^\Theta\\&=\frac{V_{g-1,n}^\Theta}{4}+\mathit{O}\left(\frac{V_{g,n+1}^\Theta}{g^4}\right)\\&=\frac{V_{g-1,n}^\Theta}{4}+\mathit{O}\left(\frac{V_{g,n}^\Theta}{g^3}\right).
\end{align*}
\eqref{Vratio2} and $e^i_{n,\mathbf{d}}$,$a_n^i,b_n^i$ for $i=0,1$ in \eqref{tauasymp2}-\eqref{Gasymp2} imply that $g\to\infty$,\begin{equation}\label{g-1}\frac{V_{g-1,n}^\Theta}{V_{g,n}^\Theta}=\frac{1}{g^2}\times\frac{1}{4\pi^2}+\mathit{O}\left(\frac{1}{g^3}\right).\end{equation}
Therefore, as $g\to\infty$,
\begin{equation}\label{V32}
\Big[\frac{1}{g^2}\Big]V_3=\frac{1}{2\pi^3}-\frac{\delta_{k,0}}{8\pi^2},
\end{equation}
where $\delta_{k,0}=1$ when $k=0$ and otherwise $\delta_{k,0}=0$.

Hence, by \eqref{V12}, \eqref{V22} and \eqref{V32}, we get
\begin{align*}
&\left[\frac{1}{g^2}\right]\left(\frac{4^{|\mathbf{d}|}\left([\tau_{d_1}\cdots\tau_{d_s}\tau_{k}\tau_0^{n-s-1}]^\Theta_{g}-4[\tau_{d_1}\cdots\tau_{d_s}\tau_{k+1}\tau_0^{n-s-1}]^\Theta_{g}\right)}{V_{g,n}^\Theta}\right)\\&=-\frac{1}{\pi^4}\Bigg[32|\mathbf{d}|^3+(48n-72)|\mathbf{d}|^2+\Big(16n^2+(4\pi^2-48)n-12\pi^2+36\Big)|\mathbf{d}|\\&\qquad+(2\pi^2-4)n^2+(12-7\pi^2)n-8+3\pi^2-\frac{\pi}{2}\Bigg]-\frac{\delta_{k,0}}{8\pi^2}.
\end{align*}
Then by \eqref{summation}, for $\mathbf{d}=(d_1,...,d_s,0,...,0)$ with $1\leq s\leq n$ and $d_i\geq 1$ for $1\leq i\leq s$, we have
\begin{equation}\label{e2}
\begin{aligned}
e^2_{n,\mathbf{d}}&=-\frac{1}{\pi^4}\Bigg[8|\mathbf{d}|^4+(16n-40)|\mathbf{d}|^3+\Big(8n^2+(2\pi^2-38)n-6\pi^2+62\Big)|\mathbf{d}|^2\\&\qquad+\Big((2\pi^2-12)n^2-(9\pi^2-34)n+9\pi^2-\frac{\pi}{4}-38\Big)|\mathbf{d}|\Bigg]-\frac{s}{16\pi^2}.
\end{aligned}
\end{equation}
By \eqref{Nre} and Remark 3.3, we have
\begin{equation}\label{an2}\begin{aligned}
a_n^2=\frac{1}{\pi^4}\Bigg[\Big(-\frac{3\pi^2}{2}+8\Big)n^2-\Big(\frac{\pi^4}{8}-5\pi^2+20\Big)n+\frac{17\pi^4}{64}-\frac{\pi^3}{16}-\frac{27\pi^2}{8}+\frac{\pi}{4}+12\Bigg].
\end{aligned}
\end{equation}
By \eqref{Vratio2}, $a_n^i,b_n^i$ for $i=0,1$ in \eqref{Nasymp2} and \eqref{Gasymp2} and \eqref{e2}, we get $b_{n+2}^2$ for $n\geq 0$ from \eqref{Gre}
\begin{equation}\label{bn22}
b_{n+2}^2=\frac{1}{\pi^4}\left[(2\pi^2-4)n^2+(\pi^2-4)n-\frac{35\pi^2}{8}-\frac{\pi}{2}\right]
\end{equation}
By $a_n^i,b_n^i$ for $i=1,2$, \eqref{n1} and \eqref{n0}, we get
$$b_1^2=\frac{1}{\pi^4}\left(-\frac{27\pi^2}{8}-\frac{\pi}{2}\right),\ \ \text{and}\ \ b_0^2=\frac{1}{\pi^4}\left(\frac{13\pi^2}{8}-\frac{\pi}{2} - 8\right).$$
There, for $n\geq 0$, we have
\begin{equation}\label{bn2}
b_{n}^2=\frac{1}{\pi^4}\left[(2\pi^2-4)n^2-(7\pi^2-12)n+\frac{13\pi^2}{8}-\frac{\pi}{2}-8\right].
\end{equation} \qed

\vskip 20pt
\section{Polynomial property of cofficients in the expansions}
\setcounter{equation}{0}
In this section, we prove Theorem \ref{Polythm} via the Algorithm in Section 3.1. For the convience, the proof will be given in two steps. First, we prove the existence of $e^i_{n,\mathbf{d}}$, $a^i_n$ and $b^i_n$ and then we prove their polynomial properties. In fact, we can prove a more stronger statement for the existence of $e^i_{n,\mathbf{d}}$, that is, each coefficient $e^i_{n,\mathbf{d}}$ is bounded by some polynomials. The existence of these polynomials as well as the existence of $a^i_n$ and $b^i_n$ can be proved by induction on $i$ via the Algorithm in Section 3.1.

\begin{theorem}As $g\to\infty$,\\
(1) For any given $s\geq 0$, $n\geq 1$, there exist polynomials $T_{n}^s(d_1,...,d_n)$ and $t_{n}^s(d_1,...,d_n)$ in variables $d_1,...,d_n$ of degree $2s+2$ and $2s$ respectively such that for any $\mathbf{d}=(d_1,...,d_n)$, one has
\begin{equation}\label{ePoly1}
\bigg|\frac{4^{|\mathbf{d}|}[\tau_{d_1}\cdots\tau_{d_n}]_g^\Theta}{V_{g,n}^\Theta}-1-\frac{e_{n,\mathbf{d}}^1}{g}-\cdots-\frac{e_{n,\mathbf{d}}^s}{g^s}\bigg|\leq\frac{T_{n}^s(d_1,...,d_n)}{g^{s+1}},
\end{equation}
and
\begin{equation}\label{ePoly2}
|e_{n,\mathbf{d}}^s|\leq t_{n}^s(d_1,...,d_n).
\end{equation}
(2) For any given $n,s\geq 0$, there exist $a_n^i$ and $b_n^i$ for $i=1,...,s$ independent of $g$ such that
\begin{equation}\label{Nasymp4.1}
\frac{\frac{\pi}{2}(2g-2+n)V^\Theta_{g,n}}{V^\Theta_{g,n+1}}=1+\frac{a^1_{n}}{g}+\cdots+\frac{a^s_{n}}{g^s}+\mathit{O}\left(\frac{1}{g^{s+1}}\right),
\end{equation}
\begin{equation}\label{Gasymp4.1}
\frac{4V^\Theta_{g-1,n+2}}{V^\Theta_{g,n}}=1+\frac{b^1_{n}}{g}+\cdots+\frac{b^s_{n}}{g^s}+\mathit{O}\left(\frac{1}{g^{s+1}}\right).
\end{equation}
\end{theorem}

By \eqref{summation} and Fact 4 (3) (cf. Appendix A), Theorem 4.1 (1) is equivalent to the following statement.

\begin{proposition}
There exist polynomials $\widetilde{T}_{n}^s(d_1,...,d_n)$ and $\widetilde{t}_{n}^s(d_1,...,d_n)$ in variables $d_1,...,d_n$ of degrees $2s+1$ and $2s-1$ respectively such that for any $\mathbf{d}=(d_1,...,d_n)$, one has
\begin{equation}\label{ePoly11}
\bigg|\frac{4^{|\mathbf{d}|}\left([\tau_{d_1}\tau_{d_2}\cdots\tau_{d_n}]_g^\Theta-4[\tau_{d_1+1}\tau_{d_2}\cdots\tau_{d_n}]_g^\Theta\right)}{V_{g,n}^\Theta}-\frac{\widetilde{e}_{n,\mathbf{d}}^1}{g}-\cdots-\frac{\widetilde{e}_{n,\mathbf{d}}^s}{g^s}\bigg|\leq\frac{\widetilde{T}_{n}^s(d_1,...,d_n)}{g^{s+1}},
\end{equation}
and
\begin{equation}\label{ePoly21}
|\widetilde{e}_{n,\mathbf{d}}^s|\leq \widetilde{t}_{n}^s(d_1,...,d_n).
\end{equation}
\end{proposition}

\noindent\textbf{Proof of Theorem 4.1.} We use induction on $r$ to prove Theorem 4.1. Moreover, we need the following two lemmas.
\begin{lemma}
\eqref{Nasymp4.1}, \eqref{Gasymp4.1}, \eqref{ePoly11} and \eqref{ePoly21} for $s<r$ impliy \eqref{ePoly11} and \eqref{ePoly21} for $s=r$.
\end{lemma}

\begin{lemma}
\eqref{ePoly11}, \eqref{ePoly21} for $s=r$ and \eqref{Nasymp4.1}, \eqref{Gasymp4.1} for $s<r$ implies \eqref{Nasymp4.1}, \eqref{Gasymp4.1} for $s=r$.
\end{lemma}

\noindent Assume that two lemmas above hold, then Theorem 4.1 follows by Lemma 4.3 and Lemma 4.4,  since the cases $s=0,1$ hold in view of Theorem 3.1.\qed

\noindent\textbf{Proof of Lemma 4.3.}
By \eqref{difference}, we deduce that \eqref{ePoly11} and \eqref{ePoly21} for $s=r$ is equivalent to obtain the expansions $V_1$, $V_2$ and $V_3$ up to $\mathit{O}(1/g^{r+1})$. Moreover, it reduces to to get the expansions of $\displaystyle\frac{4^{|\mathbf{d}|}A_{\mathbf{d},g,n}}{V_{g,n-1}^\Theta},\ \frac{4^{|\mathbf{d}|-1}B_{\mathbf{d},g,n}}{V_{g-1,n+1}^\Theta},\  \frac{4^{|\mathbf{d}|}C_{\mathbf{d},g,n}}{V_{g,n}^\Theta}$ up to $\mathit{O}(1/g^{r})$ in view of \eqref{V1}-\eqref{V3} and that \eqref{Nasymp4.1}-\eqref{ePoly21} holds for $s<r$. For example, \eqref{AV} now reads
\begin{equation}\label{Apoly}\frac{4^{|\mathbf{d}|}A_{\mathbf{d},g,n}}{V_{g,n-1}^\Theta}=\sum\limits_{j=2}\limits^{n}(2d_j+1)\sum\limits_{L=0}\limits^{\infty}\Big(\frac{a_{L}}{4^{L}}-\frac{a_{L-1}}{4^{L-1}}\Big)\left(1+\frac{e_{n,\mathbf{d}(L,j)}^1}{g}+\cdots+\frac{e_{n,\mathbf{d}(L,j)}^{r-1}}{g^{r-1}}+E_{\mathbf{d}(L,j),r}\right),\end{equation}
where $\mathbf{d}(L,j)=(L+d_1+d_j,d_2,...,\widehat{d_j},...,d_n)$ and $|E_{\mathbf{d}(L,j),r}|\leq T_n^{r-1}\left(\mathbf{d}(L,j)\right)/g^r$.
So by Fact 4 (2) (cf. Appendix A), for $s<r$, $$\widetilde{e}_{n,\mathbf{d}}^s=\sum\limits_{j=2}\limits^{n}(2d_j+1)\sum\limits_{L=0}\limits^{\infty}\left(\frac{a_L}{4^L}-\frac{a_{L-1}}{4^{L-1}}\right)e_{n,\mathbf{d}(L,j)}^{s-1}$$ is again 
a polynomial in varibles $d_1,...,d_n$ of degree $2s-1$ with the help of Lemma \ref{Coeff-est} (4).
Therefore, the expansion of $\displaystyle\frac{4^{|\mathbf{d}|}A_{\mathbf{d},g,n}}{V_{g,n-1}^\Theta}$ up to $\mathit{O}(1/g^{r})$ is derived. Hence, we get the expansion of $V_1$ up to $\mathit{O}(1/g^{r+1})$ by \eqref{Nasymp4.1}-\eqref{ePoly21} holding for $s<r$ via \eqref{V1}. The expansions of other two terms $V_2$ and $V_3$ can be obtained in a similar way.\qed

\noindent\textbf{Proof of Lemma 4.4.} Let $e^s_{n,L}$ denotes $e^s_{n,\mathbf{d}}$ for $\mathbf{d}=(L,0,...,0)$. By Fact 1 (cf. Appendix A) and \eqref{ePoly11}-\eqref{ePoly21} for $s=r$, as $g\to\infty$, \eqref{Nre} implies
\begin{equation}\label{Ncalcu}\frac{\frac{\pi}{2}(2g-2+n)V_{g,n}^\Theta}{V_{g,n+1}^\Theta}=1+\frac{a_n^1}{g}+\cdots+\frac{a_n^{r}}{g^{r}}+\mathit{O}\left(\frac{1}{g^{r+1}}\right),\end{equation}
where for any $1\leq i\leq r$\begin{equation}\label{Ncalcu2}a_n^i=\sum\limits_{L=0}\limits^{\infty}\frac{(-1)^L(\pi/2)^{2L+1}}{(2L+1)!}e^i_{n,L}.\end{equation}
The existence of $a_n^i$ is guaranteed by Remark 3.2 and that $e^i_{n,L}$ is a polynomial in $L$ of degree $2i$.

Similarly, we can use \eqref{Gre} to calculate $b_n^i$ for $1\leq i\leq r$ in the expansion of $\frac{4V_{g-1,n+4}^\Theta}{V_{g,n+2}^\Theta}$. We apply \eqref{ePoly11} and \eqref{ePoly21} for $s=r$ to get the expansion of $\frac{4[\tau_1\tau_0^{n+1}]_g^\Theta}{V_{g,n+2}^\Theta}$ up to $\mathit{O}\left(\frac{1}{g^{r+1}}\right)$, then we shall consider the second term in the right hand side of \eqref{Gre}. Note that by Lemma \ref{Vproduct},
\begin{equation}\label{Gcalcu2}\begin{aligned}&\frac{4}{V_{g,n+2}^\Theta}\sum_{\substack{g_1+g_2=g\\[3pt]I\sqcup J=\{1,...,n\}}}V^{\Theta}_{g_1,|I|+2}V^{\Theta}_{g_2,|J|+2}\\&=8\sum\limits_{2j+i\leq r}\binom{n}{i}\frac{V_{g-j,n+2-i}^\Theta}{V_{g,n+2}^\Theta}\times V_{j,i+2}^\Theta+\mathit{O}\left(\frac{1}{g^{r+1}}\right).\end{aligned}\end{equation}
Then by \eqref{Nasymp4.1} and \eqref{Gasymp4.1} for $s<r$, \eqref{Vratio2} yields the expansion of $\frac{V_{g-j,n+2-i}^\Theta}{V_{g,n+2}^\Theta}$ up to $\mathit{O}\left(\frac{1}{g^{r+1}}\right)$. So we get the expansion of $\frac{4V_{g-1,n+4}^\Theta}{V_{g,n+2}^\Theta}$ up to $\mathit{O}\left(\frac{1}{g^{r+1}}\right)$. The remaining cases for $\frac{4V_{g-1,3}^\Theta}{V_{g,1}^\Theta}$ and $\frac{4V_{g-1,2}^\Theta}{V_{g,0}^\Theta}$ can be obtained by \eqref{n1} and \eqref{n0}, respectively. \qed

\noindent\textbf{Proof of Theorem \ref{Polythm}.} Two lemmas as below are needed. Also, here we assume they hold and the proofs of them will be given later.

\begin{lemma}
(1) For given $s$, there exist polynomials $q_j(d_1,...,d_n)$, j=0,...,s such that $$e^s_{n,\mathbf{d}}=\sum\limits_{j=0}\limits^{s}e_{\mathbf{d},j}n^j,$$
where $|e_{\mathbf{d},j}|\leq q_j(d_1,...,d_n)$.\\
(2) Each $a_n^s$, $b_n^s$ in \eqref{Nasymp4.1}, \eqref{Gasymp4.1} respectively is a polynomial in $n$ of degree $s$.
\end{lemma}

\begin{lemma}
(1) Given $s$ and for any fixed $n$, $e^s_{n,\mathbf{d}}$ is a polynomial in $\mathbb{R}[d_1,...,d_n]$ of degree $2s$. Moreover, the coefficient $[d_1^{\alpha_1}\cdots d_n^{\alpha_n}]e^s_{n,\mathbf{d}}$ is a linear rational combination of $\pi^{-2\lfloor \frac{|\bm{\alpha}|+1}{2}\rfloor}$, $\pi^{-2\lfloor \frac{|\bm{\alpha}|+1}{2}\rfloor-1}$,...,$\pi^{-2s}$, where $|\bm{\alpha}|=\alpha_1+\cdots+\alpha_n\leq 2s$.\\
(2) Each $a_n^s$, $b_n^s$ in \eqref{Nasymp4.1}, \eqref{Gasymp4.1} respectively is a polynomial in $\mathbb{Q}[\pi^{-1}]$ of degree $2s$.
\end{lemma}

\noindent The existence of $e_{n,\mathbf{d}}^i$, $a_n^i$ and $b_n^i$ in \eqref{Tauasymp}-\eqref{Gasymp} is guaranteed by Theorem 4.1. The polynomial property of $e_{n,\mathbf{d}}^i$ follows by Lemma 4.6 (1). Also, the polynomial properties of $a_n^i$ and $b_n^i$ hold by Lemma 4.5 (2) and Lemma 4.6 (2).\qed

Now we prove Lemma 4.5 and Lemma 4.6 to complete the proof of Theorem 1.2.

\noindent\textbf{Proof of Lemma 4.5.}
The proof is again by induction on $s$. The cases $s=0,1$ for part (1) and part (2) hold in view of Theorem 3.1. Similar to the proof of Lemma 4.3 and Lemma 4.4, now by \eqref{Nre} and \eqref{Gre}, part (1) for $s=r$ implies part (2) for $s=r$ in view of \eqref{Ncalcu2} and \eqref{Gcalcu2}. Moreover, part (2) for $s<r$ and part (1) for $s<r$ imply part (1) for $s=r$. This follows from \eqref{summation} and Fact 4 (3) ((cf. Appendix A) by analyzing the contributions from $V_1$-$V_3$ in \eqref{difference} via \eqref{AV}-\eqref{CV}.\qed

\begin{remark}
Lemma 4.5 (1) shows that $e^s_{n,\mathbf{d}}$ can be viewed as a polynomial in $n$ with coefficients grow at most polynomially in $d_1,...,d_n$. So together with Fact 4 (2) (cf. Appendix A), we have for $0\leq j\leq s$,
$$\sum\limits_{L=0}\limits^{\infty}\Big(\frac{a_{L}}{4^L}-\frac{a_{L-1}}{4^{L-1}}\Big)e_{\mathbf{d}(L,i),j}<\infty,$$
where $\mathbf{d}(L,i)=(d_1+d_i+L,...,\hat{d_i},...,d_n).$
\end{remark}

\noindent\textbf{Proof of Lemma 4.6.} Similar to the equivalence between Theorem 4.1 (1) and Proposition 4.2 in view of \eqref{summation} and Fact 4 (3) (cf. Appendix A), Lemma 4.6 (1) is equivalent to the same statement for $\widetilde{e}^s_{n,\mathbf{d}}$ as below.
\begin{proposition} %4.8
Given $s$ and for any fixed $n$, $\widetilde{e}^s_{n,\mathbf{d}}$ is a polynomial in $\mathbb{R}[d_1,...,d_n]$ of degree $2s-1$. Moreover, the coefficient $[d_1^{\alpha_1}\cdots d_n^{\alpha_n}]\widetilde{e}^s_{n,\mathbf{d}}$ is a linear rational combination of $\pi^{-2\lfloor \frac{|\bm{\alpha}|+1}{2}\rfloor}$,$\pi^{-2\lfloor \frac{|\bm{\alpha}|+1}{2}\rfloor-1}$,...,$\pi^{-2s}$, where $|\bm{\alpha}|=\alpha_1+\cdots+\alpha_n\leq 2s-1$.
\end{proposition}
So we only need to prove Proposition 4.8 and Lemma 4.6 (2). The proof relies on induction on $s$ with the same techniques which are used in the proof of Theorem 4.1. We need the following three claims:

\begin{claim}
Part (2) of Lemma 4.6 for $s<r$ implies that $\displaystyle \left[\frac{1}{g^r}\right]\frac{V_{g_1,n_1+1}^\Theta\cdot V_{g-g_1,n-n_1+1}^\Theta}{V_{g,n}^\Theta}$ is a polynomial of degree at most $2r$ in $\mathbb{Q}[\pi^{-1}]$. Moreover, when $g_1$, $n_1$ and $s$ are fixed
$$\frac{V_{g_1,n_1+1}^\Theta\cdot V_{g-g_1，n-n_1+1}^\Theta}{V_{g,n}^\Theta}=\sum\limits_{k=2g_1+n_1+1}\limits^{s}\frac{c_{g_1,n_1}^k}{g^k}+\mathit{O}\left(\frac{1}{g^{s+1}}\right),$$
where $c_{g_1,n_1}^k$ is a polynomial of degree at most $2s$ in $\mathbb{Q}[\pi^{-1}]$.
\end{claim}

\begin{claim}
Part (1) of Lemma 4.6 for $s=r$ and part (2) of Lemma 4.6 for $s<r$ imply part (2) of Lemma 4.6 for $s=r$.
\end{claim}

\begin{claim}
Part (1)-(2) of Lemma 4.6 for $s<r$ imply Proposition 4.8 for $s=r$.
\end{claim}
Assume the three claims above hold. Note that Lemma 4.6 (1) is equivalent to Proposition 4.8 and Lemma 4.6 for $s=0,1$ hold in view of Theorem 3.1. So Lemma 4.6 follows by Claim 4.9-4.11.\qed

Now we prove the three claims above.

\noindent\textbf{Proof of Claim 4.9.}
This is a simple consequence of \eqref{Vratio2} together with Lemma 4.6 (2) holding for $s<r$.\qed

\noindent\textbf{Proof of Claim 4.10.}
First, we analyze $a_n^r$ by \eqref{Ncalcu2}. Note that Lemma 4.6 (1) for $s=r$ implies that each element $[L^i]e^{r}_{n,L}$ is a linear rational combination of $\pi^{-2\lfloor\frac{i+1}{2}\rfloor}$,...,$\pi^{-2r}$ for $0\leq i\leq 2s$. So in view of Remark 3.2, we deduce that $a_n^r$ is a rational polynomial of degree at most $2r$ in $\mathbb{Q}[\pi^{-1}]$. The similar statement for $b_n^r$ follows from \eqref{Gre}, Lemma 4.6 (1) for $s=r$, $\mathbf{d}=(1,0,...,0)$ and Claim 4.9. \qed

\noindent\textbf{Proof of Claim 4.11.}
Following the fourth step of the Algorithm in Section 3.1, now we analyze $$\frac{4^{|\mathbf{d}|}\left([\tau_{d_1}\tau_{d_2}\cdots\tau_{d_n}]_g^\Theta-4[\tau_{d_1+1}\tau_{d_2}\cdots\tau_{d_n}]_g^\Theta\right)}{V_{g,n}^\Theta}$$ by evaluating contributions from $V_1$, $V_2$, $V_3$. Here $\widetilde{e}^{0,i}_{n,\mathbf{d}}=0$ and the expansion of $V_i$ is $$V_i=\frac{\widetilde{e}^{1,i}_{n,\mathbf{d}}}{g}+\cdots+\frac{\widetilde{e}^{s,i}_{n,\mathbf{d}}}{g^{s}}+\mathit{O}\left(\frac{1}{g^{s+1}}\right).$$

\noindent$\bullet$\textbf{Contribution from $V_1$.}
By \eqref{V1} and \eqref{Apoly}, we get $$\widetilde{e}^{r,1}_{n,\mathbf{d}}=\sum\limits_{j_1+j_2=r}q^{j_1,1}(\mathbf{d})\widetilde{a}_{n-1}^{j_2},$$
where $$q^{j_1,1}(\mathbf{d})=\sum\limits_{j=2}\limits^{n}(2d_j+1)\sum\limits_{L=0}\limits^{\infty}\left(\frac{a_L}{4^{L}}-\frac{a_{L-1}}{4^{L-1}}\right)e^{j_1-1}_{n,\mathbf{d}(L,j)}$$
with $\mathbf{d}(L,j)=(d_1+d_j+L,d_2,...,\widehat{d}_j,...,d_n)$ and $$\widetilde{a}_{n-1}^{j_2}=\left[\frac{1}{g^{j_2}}\right]\left(\frac{1}{\frac{\pi}{2}(2g-3+n)}\cdot\frac{\frac{\pi}{2}(2g-3+n)V_{g,n-1}^\Theta}{V_{g,n}^\Theta}\right).$$
By part (2) of Lemma 4.6 for $s<r$, $\widetilde{a}_{n-1}^{j_2}$ for $1\leq j_2\leq r$ is a rational linear combination of $\pi^{-1}$,$\pi^{-2}$,...,$\pi^{-2j_2+1}$.

On the other hand, by part (1) of Lemma 4.6 for $s<r$ and Fact 4 (2) (cf. Appendix A), we deduce that for $j_1\leq r$, $q^{j_1,1}(\mathbf{d})$ is again a polynomial in $\mathbb{R}[d_1,...,d_n]$ of degree $2j_1-1$. Therefore, $$\widetilde{e}^{r,1}_{n,\mathbf{d}}=\sum\limits_{j_1+j_2=r}q^{j_1,1}(\mathbf{d})\widetilde{a}_{n-1}^{j_2}$$
is also a polynomial in $\mathbb{R}[d_1,...,d_n]$ of $2j_1-1$.

By part (1) of Lemma 4.6 for $s<r$, the coefficient $[d_1^{\alpha_1}\cdots d_n^{\alpha_n}]q^{j_1,1}(\mathbf{d})$ is a rational linear combination of the form  with $\mathbf{x}=(d_1+d_j+L,d_2,...,\hat{d_j},...,d_n)$ $$c_{\bm{\alpha}(j,r)}\sum\limits_{L=0}\limits^{\infty}L^r\left(\frac{a_L}{4^L}-\frac{a_{L-1}}{4^{L-1}}\right),$$
where $c_{\bm{\alpha}(j,r)}=[x_1^{r+\alpha_1+\alpha_j}x_2^{\alpha_2}\cdots\widehat{x_j^{\alpha_j}}\cdots x_n^{\alpha_n}]e^{j_1-1}_{n,\mathbf{x}}$
and $$d_{\bm{\alpha}(j,r)}\sum\limits_{L=0}\limits^{\infty}L^r\left(\frac{a_L}{4^L}-\frac{a_{L-1}}{4^{L-1}}\right),$$
where $d_{\bm{\alpha}(j,r)}=[x_1^{r+\alpha_1+\alpha_j-1}x_2^{\alpha_2}\cdots\widehat{x_j^{\alpha_j}}\cdots x_n^{\alpha_n}]e^{j_1-1}_{n,\mathbf{x}}$.
Note that Lemma \ref{Coeff-est} (2)-(4) implies that for $r\in\mathbb{Z}_{\geq 0}$, $$\pi\sum\limits_{L=0}\limits^{\infty}L^r\left(\frac{a_L}{4^L}-\frac{a_{L-1}}{4^{L-1}}\right)$$ is a rational linear combination of $\pi^{2\lfloor r/2\rfloor}$, $\pi^{2\lfloor r/2\rfloor-2}$,...,$1$. Moreover, part (1) of Lemma 4.6 for $s<r$ implies that $c_{\bm{\alpha}(j,r)}$ is a rational linear combination of $\pi^{-2\lfloor \frac{|\bm{\alpha}|+1+r}{2}\rfloor}$,$\pi^{-2\lfloor \frac{|\bm{\alpha}|+1+r}{2}\rfloor-1}$,...,$\pi^{-2j_1+2}$ and $d_{\bm{\alpha}(j,r)}$ is a rational linear combination of $\pi^{-2\lfloor \frac{|\bm{\alpha}|+r}{2}\rfloor}$,$\pi^{-2\lfloor \frac{|\bm{\alpha}|+r}{2}\rfloor-1}$,...,$\pi^{-2j_1+2}$. Therefore, $[d_1^{\alpha_1}\cdots d_n^{\alpha_n}]q^{j_1,1}(\mathbf{d})$ is a rational linear combination of $\pi^{-2\lfloor \frac{|\bm{\alpha}|+1}{2}\rfloor}$,$\pi^{-2\lfloor \frac{|\bm{\alpha}|+1}{2}\rfloor-1}$,...,$\pi^{-2j_1+1}$. Hence, combining with the information of $\widetilde{a}^{j_2}_{n-1}$ for $1\leq j_2\leq r$, it follows that $[d_1^{\alpha_1}\cdots d_n^{\alpha_n}]\widetilde{e}^{r,1}_{n,\mathbf{d}}$ is a rational combination of $\pi^{-2\lfloor \frac{|\bm{\alpha}|+1}{2}\rfloor}$,$\pi^{-2\lfloor \frac{|\bm{\alpha}|+1}{2}\rfloor-1}$,...,$\pi^{-2r}$.

Therefore, we already proved that $\widetilde{e}^{r,1}_{n,\mathbf{d}}$ satisfies Proposition 4.8 for $s=r$. 

\noindent$\bullet$\textbf{Contribution from $V_2$.} In the similar way, we use \eqref{V2} and part (1)-(2) of Lemma 4.6 for $s<r$ to get the expansion of $\widetilde{e}^{r,2}_{n,\mathbf{d}}$.Then we use Lemma \ref{Coeff-est} and Lemma 4.6 (1)-(2) for $s<r$ to verify the properties of $\widetilde{e}^{r,2}_{n,\mathbf{d}}$ stated in Proposition 4.8 for $s=r$ via the same method used in the former part.

\noindent$\bullet$\textbf{Contribution from $V_3$.} Similarly, in this case, we use \eqref{V3} to get the expansion of $\widetilde{e}^{r,3}_{n,\mathbf{d}}$. Note that Lemma \ref{Vproduct} implies that
$$\frac{\sum_{\substack{1\leq g^{'}\leq g-1\\[3pt]I\sqcup J=\{2,...,n\}}} V_{g^{'},|I|+1}^\Theta\cdot V_{g-g^{'},|J|+1}^\Theta}{V_{g,n}^\Theta}$$
is of order at least $\mathit{O}(1/g^{2})$ and it is of order $(1/g^{r+1})$ unless $ 2g^{'}+|I|\leq r$ or $2(g-g^{'})+|J|\leq r$.
Then by \eqref{Vratio2} and part (1)-(2) of Lemma 4.6 for $s<r$, we obtain the expansion of $\widetilde{e}^{r,3}_{n,\mathbf{d}}$. In the end, by Lemma \ref{Coeff-est} and Lemma 4.6 (1)-(2) for $s<r$, we check the properties of $\widetilde{e}^{r,3}_{n,\mathbf{d}}$ stated in Proposition 4.8 for $s=r$ via the same method in analyzing the contribution from $V_1$.\qed

\begin{remark} From Claim 4.11, we point out the polynomial properties of $V_i$ for $i=1,2,3$, that is, there exist polynomials $\widetilde{p}_{s,i}$, $i=1,2,3$ such that:
\begin{itemize}[leftmargin=2em]
\item [(1)] For any fixed $n$, the coefficient $[1/g^s]V_i=\widetilde{p}_{s,i}(d_1,...,d_n)$;
\item [(2)] The coefficient $[d_1^{\alpha_1}\cdots d_n^{\alpha_n}]\widetilde{p}_{s,i}$ is a linear rational combination of $\pi^{-2\lfloor\frac{|\bm{\alpha}|+1}{2} \rfloor}$,\\$\pi^{-2\lfloor\frac{|\bm{\alpha}|+1}{2} \rfloor-1}$,...,$\pi^{-2s}$, where $|\bm{\alpha}|=\alpha_1+\cdots+\alpha_n$.
\end{itemize}
\end{remark}

\vskip 20pt
\section{Proof of Theorem \ref{superWP} and Theorem \ref{GPRSConj}}
\setcounter{equation}{0}
From Theorem \ref{Polythm} and Remark \ref{Coeffs}, we get
\begin{corollary}
Given $n\geq 1$, there exist $g_0$ and $g_1$ such that the following sequence  $$\bigg\{\frac{4V_{g-1,n+2}^\Theta}{V_{g,n}^\Theta}\bigg\}_{g\geq g_0}.$$is increasing and the sequence $$\bigg\{\frac{\frac{\pi}{2}(2g-2+n)V^\Theta_{g,n}}{V_{g,n+1}^\Theta}\bigg\}_{g\geq g_1}$$ is decreasing.
\end{corollary}

Since
\begin{equation}\label{Vratio}
\frac{V_{g+1,n}^{\Theta}}{V_{g,n}^{\Theta}}=\frac{V_{g+1,n}^{\Theta}}{4V_{g,n+2}^{\Theta}}\cdot\frac{V_{g,n+2}^{\Theta}}{\frac{\pi}{2}(2g-1+n)V_{g,n+1}^{\Theta}}\cdot\frac{V_{g,n+1}^{\Theta}}{\frac{\pi}{2}(2g-2+n)V_{g,n}^{\Theta}}\pi^2(2g-1+n)(2g-2+n),
\end{equation}
by Remark \ref{Coeffs} and Fact 3 (cf. Appendix A), \eqref{Nasymp} and \eqref{Gasymp} immediately yield
\begin{corollary}
Let $n\geq 0$ be fixed, then $$\frac{V_{g+1,n}^{\Theta}}{V_{g,n}^{\Theta}}=\pi^2(2g-1+n)(2g-2+n)\bigg(1-\frac{1}{2g}\bigg)\cdot\bigg(1+\frac{c_n^2}{g^2}+\cdots+\frac{c_n^k}{g^k}+\mathit{O}\Big(\frac{1}{g^{k+1}}\Big)\bigg),$$ as $g\to\infty$, where $$c_n^2=\frac{n}{4}-\frac{7}{32}+\frac{1}{8\pi}+\frac{1}{\pi^2}\left(n^2-3n+\frac{25}{8}\right).$$
\end{corollary}

Since $\prod_{g=1}^\infty(1+a_g)$ converges when $a_g=\mathit{O}(1/g^2)$. Therefore, Corollary 5.2 implies that there exists a universal constant $C\in (0,\infty)$ such that
\begin{equation}\label{Const}
\lim\limits_{g\to\infty}\frac{2^n\sqrt{g}V_{g,n}^{\Theta}}{(2g-3+n)!\pi^{2g+n}}=C.
\end{equation}
Moreover, in view of Fact 2 (cf. Appendix A), we get Theorem  \ref{superWP}. 

\textbf{Proof of Theorem \ref{superWP}} For $n=0$, we have
$$V_{g,0}^\Theta=V_{2,0}^\Theta\prod\limits_{j=2}\limits^{g-1}\frac{V_{j+1,0}^\Theta}{V_{j,0}^\Theta}.$$
Then, from Corollary 5.2 and Fact 2 (cf. Appendix A), we have
\begin{equation}\label{Vg0}V_{g,0}^{\Theta}=C\frac{(2g-3)!\pi^{2g}}{\sqrt{g}}\left(1+\frac{d_0^1}{g}+\cdots+\frac{d_0^k}{g^k}+\mathit{O}\Big(\frac{1}{g^{k+1}}\Big)\right),\ \ \text{as}\ g\to\infty.\end{equation}
Denote \begin{equation}\label{Cgn}C_{g,n}:=\frac{(2g-3+n)!\pi^{2g+n}}{2^n\sqrt{g}},\end{equation} then$$C=\lim\limits_{g\to\infty}\frac{V_{g,0}}{C_{g,0}}.$$
On the other hand, 
$$\frac{V_{g,n}^{\Theta}}{C\cdot C_{g,n}}=\frac{V^\Theta_{g,0}}{C\cdot C_{g,0}}\prod\limits_{j=0}\limits^{n-1}\frac{V_{g,j+1}^\Theta}{\frac{\pi}{2}(2g-2+j)V_{g,j}}.$$
Then by Fact 5 (cf. Appendix A), Fact 3 (cf. Appendix A) and \eqref{Nasymp},
Theorem \ref{superWP} follows for $n\geq 1$. 

As for the coefficient of $a_n^i$ at $n^{2i}$, note that \eqref{Nasymp}, Remark \ref{Coeffs} and Fact 3 (cf. Appendix A) imply $l=-\frac{2}{\pi^2}$. Then by Fact 5 (cf. Appendix A), we get it.\qed

\textbf{Proof of Theorem \ref{GPRSConj}}
Recall that for any given $\alpha\in\mathbb{R}$, as $x\to\infty,$ we have
$$\Gamma(x+\alpha)\sim\Gamma(x)x^{\alpha}.$$
So when $n\geq 1$ is fixed and as $g\to\infty$, we get 
$$\frac{\Gamma \left(2g+n-\frac{5}{2}\right)\sqrt{2g}}{(2g+n-3)!}\sim\frac{\sqrt{2g}}{\sqrt{2g+n-2}}\sim1.$$
Note that $$\lim\limits_{g\to\infty}\frac{c(n)\cdot\sum\limits_{i=1}\limits^{n}L_i^2}{g}=0.$$ By Theorem \ref{Bound}, we have
$$V_{g,n}^\Theta\bigg(1-c(n)\frac{\sum\limits_{i=1}\limits^{n}L_i^2}{g}\bigg)\frac{\mathrm{sinh}(L_i/4)}{L_i/4}\leq V_{g,n}^\Theta(2L_1,...,2L_n)\leq V_{g,n}^\Theta\frac{\mathrm{sinh}(L_i/4)}{L_i/4}.$$
Therefore, part (1) follows by Theorem \ref{superWP}.

Part (2) is a direct consequence of Theorem \ref{Polythm} (1), Remark \ref{Coeffs} and Theorem \ref{superWP}.  \qed

\vskip 20pt
\section{Asymptotics for $n(g)$}
\setcounter{equation}{0}
In this section, we are interested in the aysmptotics when $n$ grows as $g\to\infty$. Here, we prove Theorem \ref{NsuperWP} when $n=\mathit{o}(\sqrt{g})$ as $g\to\infty$. 
First, we give a universal bound of $$4^{k}[\tau_{k}\tau_0^{n-1}]_g^\Theta-4[\tau_{k+1}\tau_0^{n-1}]_g^\Theta,$$
which helps to derive a uniform estimate as a generalization of Lemma \ref{Basic-est}.

\begin{proposition}
For arbitrary $n\geq 1$, we have
$$(n-1)[\tau_k\tau_0^{n-2}]_g^\Theta\leq[\tau_{k}\tau_0^{n-1}]_g^\Theta-4[\tau_{k+1}\tau_0^{n-1}]_g^\Theta\leq\frac{n+7}{4^{k-1}\pi}V_{g,n-1}^\Theta+\frac{8k+2}{4^{k-1}\pi}V_{g-1,n+1}^\Theta.$$ 
\end{proposition}
\begin{proof}
From \eqref{difference}, we denote $$[\tau_{k}\tau_0^{n-1}]^\Theta_g-4[\tau_{k+1}\tau_0^{n-1}]^\Theta_g=:A_{k,g,n}+B_{k,g,n}+C_{k,g,n},
$$
where $A_{k,g,n}-C_{k,g,n}$ are defined in \eqref{Ad}-\eqref{Cd} by setting $\mathbf{d}=(k,0,...,0)$. Now we calculate the contribution of $A_{k,g,n}$, $B_{k,g,n}$ and $C_{k,g,n}$ respectively.

\noindent$\bullet$ \textbf{Contribution from $A_{k,g,n}$.} On the one hand, Proposition 2.2 implies
\begin{equation}\label{Aestimatelow}
A_{k,g,n}=(n-1)\sum\limits_{L=0}\limits^{g-1-k}(a_L-4a_{L-1})[\tau_{k+L}\tau_0^{n-2}]_g^\Theta\geq (n-1)[\tau_k\tau_0^{n-2}]_g^\Theta.
\end{equation}
On the other hand, by Lemma \ref{Basic-est} (1) and Lemma \ref{Coeff-est} (2), we get
\begin{equation}\label{Aestimate1}
A_{k,g,n}\leq \frac{(n-1)}{4^k}\sum\limits_{L=0}\limits^{g-1-k}\Big(\frac{a_{L}}{4^{L}}-\frac{a_{L-1}}{4^{L-1}}\Big)V_{g,n-1}^\Theta\leq \frac{n-1}{4^{k-1}\pi}V_{g,n-1}^\Theta.
\end{equation}
\noindent$\bullet$ \textbf{Contribution from} $B_{k,g,n}$. Denote $\displaystyle T_{k,g,n}:=\sum\limits_{i+j=k}[\tau_i\tau_j\tau_0^{n-1}]_{g-1}^\Theta,$
then $$B_{k,g,n}=2\sum\limits_{L=0}\limits^{g-1-k}(a_{L}-4a_{L-1})T_{k+i-1,g,n}.$$
Lemma \ref{Basic-est} (1) implies that for any $k$, one has $$T_{k,g,n}\leq \frac{k+1}{4^{k}}\cdot V_{g-1,n+1}.$$
Therefore, by Lemma \ref{Coeff-est} (2)-(3), we have \begin{equation}\label{Bestimate1}B_{k,g,n}\leq\frac{2}{4^{k-1}}\sum\limits_{L=0}\limits^{\infty}(k+L)(a_{L}-4a_{L-1})V_{g-1,n+1}^\Theta= \frac{8k+2}{4^{k-1}\pi}V_{g-1,n+1}^\Theta,\end{equation}

\noindent$\bullet$ \textbf{Contribution from} $C_{k,g,n}$. Denote $$ M_{k_1,k_2,g,n}:=\sum_{\substack{g_1+g_2=g\\[3pt]I\sqcup J=\{2,...,n\}}}[\tau_{k_1}\tau_0^{|I|}]_{g_1}^\Theta\times [\tau_{k_2}\tau_0^{|J|}]_{g_2}^\Theta.$$
Note that by Lemma \ref{Basic-est} (1) and formula $(\mathbf{\Rmnum{1}})$, we have
\begin{align*}M_{k_1,k_2,g,n}\leq\frac{1}{4^{k_1+k_2}}\sum_{\substack{g_1+g_2=g\\[3pt]I\sqcup J=\{2,...,n\}}}V_{g_1,|I|+1}^\Theta\times V_{g_2,|J|+1}^\Theta\leq \frac{1}{4^{k_1+k_2}}V_{g,n-1}^\Theta.\end{align*}
Then by Lemma \ref{Coeff-est} (2), we get
\begin{equation}\label{Cestimate1}\frac{C_{k,g,n}}{V_{g,n-1}^\Theta}\leq\frac{2}{4^{k-1}}\sum\limits_{L=0}\limits^{\infty}\Big(\frac{a_{L}}{4^{L}}-\frac{a_{L-1}}{4^{L-1}}\Big)=\frac{8}{4^{k-1}\pi}.\end{equation}
Hence, this proposition follows by \eqref{Aestimatelow}- \eqref{Cestimate1}.

\end{proof}

\begin{lemma}There exist universal constant $c_0,c_1,c_2,c_3$ such that for $g,n\geq 0$, the following holds:
\begin{itemize}[leftmargin=2em]
\item [(1).] For any $k\geq 1$, $$c_0\cdot\frac{k^2}{2g-2+n}\leq 1-\frac{4^k[\tau_k\tau_0^{n-1}]_g^\Theta}{V_{g,n}^\Theta}\leq c_1\cdot\frac{nk^2}{2g-2+n}.$$
\item [(2).] $$\left|\frac{\frac{\pi}{2}(2g-2+n)V_{g,n}^\Theta}{V_{g,n+1}^\Theta}-1\right|\leq c_2\cdot\frac{n}{2g-2+n}.$$
\item [(3).] $$\frac{4V_{g-1,n+4}^\Theta}{V_{g,n+2}^\Theta}\leq 1-c_3\cdot\frac{n}{2g-2+n}.$$
\end{itemize}
\end{lemma}

\begin{proof}
Proposition 6.1 implies that 
$$4^k(n-1)[\tau_k\tau_0^{n-2}]_g^\Theta\leq 4^{k}\left([\tau_{k}\tau_0^{n-1}]_g^\Theta-4[\tau_{k+1}\tau_0^{n-1}]_g^\Theta\right)\leq\frac{4n+28}{\pi}V_{g,n-1}^\Theta+\frac{32k+8}{\pi}V_{g-1,n+1}^\Theta.$$
Since $$1-\frac{4^k[\tau_k\tau_0^{n-1}]_g^\Theta}{V_{g,n}^\Theta}=\frac{\sum\limits_{i=0}\limits^{k-1}4^i\left([\tau_i\tau_0^{n-1}]_g^\Theta-4[\tau_{i+1}\tau_0^{n-1}]_g^\Theta\right)}{V_{g,n}^\Theta},$$
then part (1) holds by \eqref{N11}, \eqref{G11} and Lemma \ref{Basic-est} (3).

Part (1) and \eqref{Nre} implies part (2).

By Lemma \ref{Vproduct} and \eqref{N11}, we have
$$\sum_{\substack{g_1+g_2=g\\[3pt]I\sqcup J=\{1,...,n\}}}V_{g_1,|I|+2}^\Theta\times V_{g_2,|J|+2}^\Theta=\mathit{O}\left(\frac{V_{g,n+4}}{g^4}\right)=\mathit{O}\left(\frac{V_{g,n+2}}{g^2}\right).$$
So we obtain part (3) from part (1) for $k=1$ and \eqref{Gre}.
\end{proof}

\noindent\textbf{Proof of Theorem \ref{NsuperWP}.}
From Lemma 5.2 (2) and (3), as $g\to\infty$,
$$\frac{\frac{\pi}{2}(2g-2+n)V_{g,n}^\Theta}{V_{g,n+1}^\Theta}\rightarrow 1,\qquad\frac{4V^\Theta_{g-1,n+4}}{V^\Theta_{g,n+2}}\rightarrow 1.$$
Since
$$\frac{V_{g,n(g)}^\Theta}{C_{g,n(g)}}=\frac{V_{g,n(g)}^\Theta}{\frac{\pi}{2}(2g-3+n)V_{g,n(g)-1}^\Theta}\cdot\frac{V_{g,n(g)-1}^\Theta}{\frac{\pi}{2}(2g-4+n(g))V_{g,n(g)-2}^\Theta}\cdot\cdots\frac{V_{g,2}^\Theta}{C_{g,2}},$$
where $C_{g,n(g)}$ is defined in \eqref{Cgn}. Then by Theorem \ref{superWP} and Lemma 5.2 (2), we get
\begin{align*}
&(1-c_2/g)\cdots(1-n(g)\cdot c_2/g)\cdot\left(C+\mathit{O}(1/g)\right)\\&\qquad\leq\frac{V_{g,n(g)}^\Theta}{C_{g,n(g)}}\leq(1+c_2/g)\cdots(1+n(g)\cdot c_2/g)\cdot\left(C+\mathit{O}(1/g)\right),
\end{align*}
which implies this theorem.\qed

\vskip 15pt
\appendix
\section{Some useful facts}
\setcounter{equation}{0}
We collect several facts that often used in the proofs of this paper.

\noindent\textbf{Fact 1.}\cite[Page 285]{mirzakhani2013growth} Let $\{r_i\}_{i=1}^\infty$ be a sequence of real numbers and $\{k_g\}_{g=1}^\infty$ be an increasing sequence of integers. Assume that for $g\geq 1$ and $i\in\mathbb{N}$, $0\leq c_{g,i}\leq c_i$ and $\lim\limits_{g\to\infty} c_{g,i}=c_i$. If $\sum\limits_{i=1}\limits^{\infty} |c_i r_i|<\infty$, then \begin{equation}\label{Fact1}\lim\limits_{g\to\infty}\sum\limits_{i=1}\limits^{k_g}r_ic_{g,i}=\sum\limits_{i=1}\limits^{\infty} r_ic_i.\end{equation}

\noindent\textbf{Fact 2.}\cite[Lemma 4.10]{mirzakhani2015towards} Let $\{c_j\}_{j=1}^\infty$ be a positive sequence with $a_2,...,a_l\in\mathbb{R}$ expressed as the following
$$c_j=1+\frac{a_2}{j^2}+\cdots+\frac{a_s}{j^s}+\mathit{O}\left(\frac{1}{j^{s+1}}\right).$$
When $g\to\infty$, there exist $b_1,...,b_{s-1}$ such that
$$\prod\limits_{j=1}\limits^{g}c_j=C_0\left(1+\frac{b_1}{g}+\cdots+\frac{b_{s-1}}{g^{s-1}}+\mathit{O}\Big(\frac{1}{g^{s}}\Big)\right),$$
where $C_0=\prod_{j=1}^{\infty}c_j$. Moreover, $b_1,...,b_{s-1}$ are polynomials in $a_2,...,a_s$ with rational coefficients.

\noindent\textbf{Fact 3.}\cite[Remark 4.4]{mirzakhani2015towards} Let $\{\omega_g\}_{g=1}^{\infty}$ be a sequence of the form $$\omega_g=1+\frac{u_1}{g}+\cdots+\frac{u_{s-1}}{g^{s-1}}+\mathit{O}\left(\frac{1}{g^s}\right),$$ then $$\frac{1}{\omega_g}=1+\frac{v_1}{g}+\cdots+\frac{v_{s-1}}{g^{s-1}}+\mathit{O}\left(\frac{1}{g^s}\right)$$ where each $v_i$ is a polynomial in $u_1,...,u_i$ with integer coefficients. Moreover, if $u_i$ is a polynomial of degree $m_i$ in $n$, then $v_k$ is a polynomial of degree at most $\text{max}_{i+j=k}(m_i+m_j)$ in $n$.

\noindent\textbf{Fact 4.}
(1) Since $\psi_i$, $\kappa_1$ and $\Theta_{g,n}\in H^2(\overline{M}_{g,n},\mathbb{Q})$ (cf. \cite{arbarello1996combinatorial,norbury2023new,wolpert1983homology}), then
$$[\tau_{d_1}\cdots\tau_{d_n}]_g^\Theta\in\mathbb{Q}\cdot\pi^{2g-2-2|\mathbf{d}|},$$
where $|\mathbf{d}|=d_1+\cdots+d_n$. In particular, $V_{g,n}^\Theta=[\tau_0^n]_g^\Theta$ is a rational multiple of $\pi^{2g-2}$.

\noindent(2) For a polynomial $p(x)=\sum_{j=1}^mb_j x^j$ of degree $m$, the polynomial
\begin{align*}\widetilde{p}(x)&=\sum\limits_{i=0}\limits^{\infty}\left(\frac{a_{i}}{4^{i}}-\frac{a_{i-1}}{4^{i-1}}\right)p(x+i)\\&=\sum\limits_{i=0}\limits^{\infty}\left(\frac{a_{i}}{4^{i}}-\frac{a_{i-1}}{4^{i-1}}\right)\sum_{j=1}^mb_j (x+i)^j\end{align*}
is again a polynomial of degree $m$. Moreover, the coefficient of $\widetilde{p}(x)$ at $x^j$ is $$[x^j]\widetilde{p}(x)=\sum\limits_{j+r\leq m}\binom{j+r}{j}b_{j+r}A(r),$$ where $$A(r)=\sum\limits_{i=0}\limits^{\infty}i^r\left(\frac{a_{i}}{4^{i}}-\frac{a_{i-1}}{4^{i-1}}\right).$$ From Lemma \ref{Coeff-est} (4), we know that $\pi A(r)$ is a polynomial in $\pi^2$ of degree $\lfloor r/2\rfloor$.

\noindent(3)\cite[Remark 4.5 (3)]{mirzakhani2015towards} Faulhaber's formula implies that the function $S_m(n)=\sum_{i=1}^ni^m$ is a polynomial in $n$ of degree $m+1$ with rational coefficients. As a result, if $\widetilde{P}(x)$ is a polynomial of degree $m$ with coefficients in a field $\mathbb{F}$, then $P(n)=\widetilde{P}(1)+\cdots+\widetilde{P}(n)$ is a polynomial of degree $m+1$ with coefficients in $\mathbb{F}$. Similarly, if $\widetilde{Q}(x,y)\in\mathbb{F}[x,y]$ is a polynomial of degree $m$, the function
$$Q(n)=\sum\limits_{i+j=n}\widetilde{Q}(i,j)$$
is again a polynomial in $\mathbb{F}[n]$.

\noindent\textbf{Fact 5.}\cite[Page 1285]{mirzakhani2015towards} Let $F(g,n)=\prod\limits_{i=1}\limits^{n}f_i(g),$ where $\{f_i\}_{i=1}^\infty$ is a sequence of functions with expansion $$f_{i}(g)=1+\frac{p(1,i)}{g}+\cdots+\frac{p(s,i)}{g^s}+\mathit{O}\left(\frac{1}{g^{s+1}}\right).$$ 
Then $$F(g,n)=1+\frac{\tilde{p}_1(n)}{g}+\cdots+\frac{\tilde{p}_s(n)}{g^s}+\mathit{O}\left(\frac{1}{g^{s+1}}\right).$$
If $p(j,k)$ is a polynomial in $k$ of degree $j$ for any given $j$, then $\tilde{p}_j(n)$ is a polynomial in $k$ of degree $2j$ for a given $j$. Moreover, $$[n^{2j}]\tilde{p}_j(n)=\frac{l^j}{2^j j!},$$ where $l=[j]p(1,j)$.

$$ \ \ \ \ $$

\end{CJK}
\end{document}